%% file: main.tex
\documentclass[12pt]{amsart}
\usepackage{amsmath}
\usepackage{amssymb}
\usepackage{fullpage}
\usepackage{algpseudocode}
\usepackage{graphics}
\usepackage{booktabs}
\usepackage{cite}
\usepackage{bbm}
\usepackage{tikz}
\usepackage{bm}
\usepackage{appendix}
\usepackage{hyperref}

\usepackage[normalem]{ulem}

\usepackage{rotating}
\usepackage{verbatim}
\usepackage{optidef}
\usepackage{cleveref}
\usepackage{array}
\usepackage{tabularx}
\usepackage{arydshln}
\usepackage{mathtools}
\usepackage{easybmat}
\usetikzlibrary{positioning}

\usepackage{enumerate}

\usepackage{longtable}
\usepackage[
  tablesonly,
  notablist,
]{endfloat}
\usepackage{rotating}
\usepackage{pdflscape}

\usepackage[ruled,vlined]{algorithm2e}

\usepackage{tikz,forest} 
\usetikzlibrary{arrows.meta}
\forestset{
    .style={
        for tree={
            base=bottom,
            child anchor=north,
            align=center,
            s sep+=1cm,
    straight edge/.style={
        edge path={\noexpand\path[\forestoption{edge},thick,-{Latex}] 
        (!u.parent anchor) -- (.child anchor);}
    },
    if n children={0}
        {tier=word, draw, thick, rectangle}
        {draw, diamond, thick, aspect=2},
    if n=1{%
        edge path={\noexpand\path[\forestoption{edge},thick,-{Latex}] 
        (!u.parent anchor) -| (.child anchor) node[pos=.2, above] {Y};}
        }{
        edge path={\noexpand\path[\forestoption{edge},thick,-{Latex}] 
        (!u.parent anchor) -| (.child anchor) node[pos=.2, above] {N};}
        }
        }
    }
}

\usepackage{tikz,pgfplots,xpatch}
\usetikzlibrary{shapes.multipart,backgrounds,shapes,fit}
\usetikzlibrary{arrows,decorations.markings,shapes.arrows,arrows.meta}

\newcommand{\R}{\mathbb{R}}
\renewcommand{\SS}{\mathbb{S}}

\newcommand{\Z}{\mathbb{Z}}
\DeclareMathOperator{\vectize}{\operatorname{vec}}
\DeclareMathOperator{\svectize}{\operatorname{svec}}

\DeclareMathOperator{\rank}{\operatorname{rank}}
\DeclareMathOperator{\inte}{\operatorname{int}}

\DeclareMathOperator{\diag}{\operatorname{diag}}
\DeclareMathOperator{\Diag}{\operatorname{Diag}}
\DeclareMathOperator{\svec}{\operatorname{svec}}
\DeclareMathOperator{\Null}{\operatorname{Null}}

\DeclareMathOperator{\Range}{\operatorname{Range}}

\DeclareMathOperator{\Pc}{\operatorname{P}}
\DeclareMathOperator{\Pct}{\operatorname{\overset{\sim}{P}}}
\DeclareMathOperator{\Pch}{\operatorname{ \widehat{P} }}

\definecolor{mypink3}{cmyk}{0, 0.7808, 0.4429, 0.1412}

\newcommand{\cE}{\mathcal{E}}
\newcommand{\cG}{\mathcal{G}}

\newcommand{\cK}{\mathcal{K}}

\newcommand{\cV}{\mathcal{V}}

\newcommand{\newgraph}{\mathcal{G}}
\newcommand{\newvertix}{\mathcal{V}}
\newcommand{\newedge}{\mathcal{E}}

\newcommand{\norm}[1]{\left\lVert#1\right\rVert}

\newcommand{\dv}[0]{\mathbf{d}}
\newcommand{\ev}[0]{\mathbf{e}}

\newcommand{\uv}[0]{\mathbf{u}}
\newcommand{\vv}[0]{\mathbf{v}}
\newcommand{\xv}[0]{\mathbf{x}}

\newcommand{\bv}[0]{\mathbf{b}}
\newcommand{\wv}[0]{\mathbf{w}}

\newcommand{\zv}[0]{\mathbf{z}}
\newcommand{\xiv}[0]{{\bm{\xi}}}
\renewcommand{\choose}[2]{\binom{#1}{#2}}
\newtheorem{thm}{Theorem}[section]

\newtheoremstyle{mytheoremstyle}{0pt}{0pt}{\itshape}{}{\bfseries}{.}{.5em}{} 

\theoremstyle{plain}
\newtheorem{cor}[thm]{Corollary}
\newtheorem{lem}[thm]{Lemma}

\newtheorem{dfn}[thm]{Definition}

\newtheorem{example}[thm]{Example}

\newtheorem{rem}[thm]{Remark} % I want the numbering, but not the italics.
\newcommand{\adj}{\mathrm{adj}}

\newcounter{algorithmctr}[section]
\renewcommand{\thealgorithmctr}{\thesection.\arabic{algorithmctr}}
    {\refstepcounter{algorithmctr}\begin{list}{}{%
        \setlength{\rightmargin}{.05\linewidth}%
        \setlength{\leftmargin}{.05\linewidth}}%
        \item[]{\setlength{\parskip}{0ex}\bigskip\par%
         \nopagebreak%
         \underline{{\bf Algorithm \thealgorithmctr.} \emph{#1.}}}}%
    {{\setlength{\parskip}{-1ex}\nopagebreak\smallskip\par} \end{list}}

\makeatletter
\newsavebox\myboxA
\newsavebox\myboxB
\newlength\mylenA

\makeatother
\newcommand\xoverline[1]{\overline{#1}}
\renewcommand\textcolor[2]{#2}
\begin{document}

\title{Computational complexity of decomposing\\ a symmetric matrix\\ as a sum of positive semidefinite and diagonal matrices}

\author{Levent Tun\c{c}el \and Stephen A. Vavasis \and Jingye Xu}
\thanks{Levent Tun\c{c}el: Department of Combinatorics and Optimization, Faculty
of Mathematics, University of Waterloo, Waterloo, Ontario N2L 3G1,
Canada (e-mail: levent.tuncel@uwaterloo.ca).
Research of this author
was supported in part by
Discovery Grants from the Natural Sciences and Engineering Research Council (NSERC) of Canada, and by U.S. Office of Naval Research under award number N00014-18-1-2078.\\
Stephen A. Vavasis: Department of Combinatorics and Optimization, Faculty
of Mathematics, University of Waterloo, Waterloo, Ontario N2L 3G1,
Canada (e-mail: vavasis@uwaterloo.ca).
Research of this author
was supported in part by a Discovery Grant from the Natural Sciences and Engineering Research Council (NSERC) of Canada.\\
Jingye Xu: 
School of
Industrial and Systems Engineering, Georgia Institute of Technology. Work of this author was completed during 2021-2022 while in the undergraduate program at David R. Cheriton School of Computer Science and the Department of Statistics and Actuarial Science, Faculty of Mathematics, University of Waterloo, Waterloo, Ontario N2L 3G1,
Canada (e-mail: jxu673@gatech.edu).
Research of this author
was supported in part by a Discovery Grant from the  Natural Sciences and Engineering Research Council (NSERC) of Canada.
}

\date{\textcolor{blue}{September 13, 2022, revised: \today}}

% name space
% 

\newcommand{\LRPD}{\textbf{LRPD}}

\begin{abstract}
We study several variants of decomposing a symmetric matrix into a sum of a low-rank positive semidefinite matrix and a diagonal matrix. Such decompositions have applications in factor analysis and they have been studied for many decades. On the one hand, we prove that when the rank of the positive semidefinite matrix in the decomposition is bounded above by an absolute constant, the problem can be solved in polynomial time. On the other hand, we prove that, in general, these problems as well as their certain approximation versions are all NP-hard. Finally, we prove that many of these low-rank decomposition problems are complete in the first-order theory of the reals; i.e., given any system of polynomial equations, we can write down a low-rank decomposition problem in polynomial time so that the original system has a solution iff our corresponding decomposition problem has a feasible solution of certain (lowest) rank. 

%In this paper, we study a low-rank plus diagonal problem ($\LRPD$)
% that arises in areas such signal process and computational finance: given a symmetric matrix $A$, decompose $A$ as a sum of a diagonal matrix and a positive semi-definite matrix with the lowest rank. We establish several results about such a problem: (i) that it is NP-hard. (ii) that algorithms involving only elementary operations exist. 
\end{abstract}

\maketitle

principal components; correspondence analysis)

\begin{section}{Introduction}
\label{sec:intro}
When a matrix is representable as a sum of a low-rank matrix and a diagonal matrix, having access to such a decomposition provides many desirable opportunities in many contexts. 
These contexts include applications in big data, machine learning, statistical analysis, numerical analysis and in the design and analysis of efficient algorithms for large-scale optimization problems. In this paper, we are interested in the computational complexity of decomposing a given symmetric matrix as a sum of a diagonal matrix and a positive semidefinite
matrix, where the rank of the positive semidefinite matrix is as small as possible.

A very useful version of this optimization (low-rank decomposition) problem has been considered for at least 100 years in the area of \emph{factor analysis} in statistics and data science (see, for instance, Albert~\cite{Albert1944}). For some more recent sources, see Bertsimas et al.~\cite{bertsimas2017certifiably}, Della Riccia and Shapiro~\cite{cite-key}, Shapiro~\cite{Shapiro2019} and the references therein. This first version of the decomposition problem, which we call $(\Pc_1)$, expresses a given covariance matrix $A$ as a sum of a low-rank positive semidefinite matrix (determining the \emph{factors}) and a diagonal matrix (determining \emph{perturbation} or \emph{noise}):

% In this paper, we study the complexity of following different versions of low-rank matrix completion:

\begin{mini*}|l|
{\dv \in \R^n} { \rank(A - \Diag(\dv))  }{}{}\tag{$\Pc_1$}
\addConstraint{ A - \Diag(\dv) \succeq 0}{}{}
\addConstraint{ \dv \geq \mathbf{0},}{}{}
\end{mini*}
where $A \in \SS^n_+$ is given.

% \begin{align*}
% & \text{Find } \dv \in \R^n, H \in \SS^{n}  \\
% & \text{subject to }&  \rank(A - \Diag(\dv) + H) = r\\ 
% & & \\ 
% & & \\ 
% \end{align*} 

In the above and the rest of the paper, $\SS^n$ denotes the space of $n$-by-$n$ symmetric matrices with real entries, $\SS^n_+$ denotes the positive semidefinite matrices in $\SS^n$. We use the trace inner product on $\SS^n$ and the usual Euclidean inner product on $\R^n$, all vectors are column vectors. $\Diag: \R^n \to \SS^n$ is the linear map $\xv \mapsto \begin{bmatrix} x_1 & 0 & \ldots & 0\\ 0 & x_2 & \ldots & 0\\ \vdots & \vdots & \ddots & \vdots \\ 0 & 0 & \ldots & x_n \end{bmatrix};$ its adjoint $\diag : \SS^n \to \R^n$ is the linear map $X \mapsto \begin{bmatrix} X_{11} \\X_{22}\\ \vdots\\X_{nn} \end{bmatrix}$. \textcolor{blue}{For a pair of matrices $A,B\in\SS^n$, we write $A \succeq B$  to mean $A-B \in \SS^n_{+}$ and} \textcolor{blue}{$A \succ B$  to mean $A-B \in \inte(\SS^n_+) =: \SS^n_{++}$.  For general matrices $A,B$, $A \otimes B$ denotes their Kronecker product $$\begin{bmatrix}
    A_{11} B &  \dots & A_{1n} B \\
    \vdots & \ddots & \vdots \\
    A_{m1} B &  \dots & A_{mn} B 
\end{bmatrix},$$ where $m,n$ are the dimensions of $A$.}

%\textbf{Notes:}

A related optimization problem considered by Saunderson et al.~\cite{Saunderson_2012} removes the non-negativity constraint on the diagonal matrix of the decomposition:

\begin{mini*}|l|
{\dv \in \R^n} { \rank(A + \Diag(\dv))  }{}{}\tag{$\Pc_2$}
\addConstraint{ A + \Diag(\dv) \succeq 0,}{}{}
\end{mini*}
where $A \in \SS^n$ is given and $\diag(A) = \mathbf{0}$ (since there are no other constraints on $\dv$, we may assume $\diag(A) = \mathbf{0}$ for convenience).

Both $(\Pc_1)$ and $(\Pc_2)$ can be regarded as low-rank matrix completion problems with unspecified diagonal entries and some additional constraints.
%For more general pattern of unspecified entries, we may use %represent it with a graph. For any undirected simple loopless graph $G = (V,E)$, we denote $n:= |V|,m := |E|$, and 
To be able to treat a more general pattern of unspecified entries, we introduce more notation. We denote by $\cG := (\cV,\cE)$ a finite, undirected, simple, \textcolor{blue}{loopless} graph with vertex set $\cV$ and edge set $\cE$. \textcolor{blue}{In this paper, all graphs are finite, undirected, loopless and simple.}
Given such a graph $\cG$, we also use the notation $\cV(\cG)$ and $\cE(\cG)$ to denote the vertices and the edges of $\cG$ respectively. When the graph is clear from the context, we simply use $\cV$ and $\cE$. $\cK_n$ denotes the clique on $n$ vertices. We also use $\cK_{\cV}$ or $\cK(\cV)$ to denote the clique with vertex set $\cV$.  For example, $\cE(\cK(\{a,b,c\}))$ denotes the set $\{\{a,b\},\{a,c\},\{b,c\}\}$.

%For the remainder of the paper, let $\Unord(S)$, where $S$ is a finite set, denote all unordered pairs $(i,j)$ with $i,j\in S$, and $i\ne j$.  Thus, $|\Unord(S)|=\choose{|S|}{2}$.  
We say $L \in \SS^n$ \emph{fits}
% a set $X\subseteq\Unord(\{1,\ldots,n\})$
an edge set $\mathcal{X}\subseteq \cE(\cK(\{1,\ldots,n\}))$
if $L_{ij} = 0$ for all $\{i,j\} \in \mathcal{X}$. With this new notation, we define another optimization problem:
\begin{mini*}|l|
{L \in \SS^n} { \rank(A + L)  }{}{}\tag{$\Pc_3$}
\addConstraint{ A + L \succeq 0}{}{}
% \addConstraint{ L_{ij} = 0, \forall ij \notin E \text{ and } i \neq j}{}{}
\addConstraint{ L \text{ fits } \mathcal{X},}{}{}
\end{mini*}
where $A \in \SS^n$ is given.  Note that in $(\Pc_3)$, similar to $(\Pc_2)$, the diagonal entries are always unspecified.  Therefore, without loss of generality, we assume $\diag(A)=\mathbf{0}$ and $A_{ij} = 0$ for all $\{i,j\} \in \cE(\cK(\{1,\ldots,n\}))\setminus \mathcal{X}$.
%(where $\bar{\cG} =(\{1,\ldots,n\}, \bar{\cE})$ is the complement of $\cG$).

%The low-rank plus diagonal problems are not quite well-studied in the literature. 

%One similar problem is low dimensional embedding. That is finding a method to map vertices of a graph into a low dimensional vector space under certain constrains. 

In all of the three optimization problems above, our objective function is always to minimize the rank of a symmetric positive semidefinite matrix variable. The underlying matrix variable is of rank $r$ iff there exist vectors $\vv_1, \vv_2, \ldots, \vv_n \in \R^r$ such that their span is $\R^r$ and the $ij^{\textup{th}}$ entry of the matrix is $\vv_i^{\top}\vv_j$. Thus, we can equivalently think about the positive semidefinite matrix variable providing a lowest-dimensional ($r$-dimensional) geometric representation given by the vectors $\vv_1, \vv_2, \ldots, \vv_n \in \R^r$. Therefore, there is a connection between the optimization problems we consider in this paper and geometric representations of graphs where the dimension of the representation space is minimized, see Laurent~\cite{Laurent1997}, Lov\'{a}sz~\cite{Lovasz2019} and the references therein. \textcolor{blue}{In such problems, either the diagonal entries of the positive semidefinite matrix variable is fixed or constrained, e.g., by conditions such as $X_{ii}=1$ or $X_{ii}+X_{jj}-2X_{ij}= 1 = \|\vv_i-\vv_j\|_2^2$. Therefore,} a key difference is that in the geometric representations of graphs, the typical constraints either fix the diagonal of the positive semidefinite matrix variable (\emph{orthonormal representations of graphs}), or constrain the diagonal by fixing or constraining the Euclidean distances between pairs of points in $\R^r$ representing the edges in a given graph (\emph{Euclidean distance matrix completion} or \emph{Euclidean graph realization}), or both (\emph{unit distance representations} contained in a hypersphere, a ball, or an ellipsoid). 
% \sout{In such problems either the diagonal entries of the positive semidefinite matrix variable is fixed or constrained, e.g., by conditions like $X_{ii}=1$ or $X_{ii}+X_{jj}-2X_{ij}= 1 = \|\vv_i-\vv_j\|_2^2$.} 

On the other hand,
our problem can also be viewed as a low-rank matrix completion problem. The minimization of nuclear norm has been used as a heuristic for such problems \cite{nuclearnorm}. Minimization of nuclear norm subject to linear equations and inequalities can be formulated as a Semidefinite Programming (SDP) problem. The recovery of an exact solution to a low-rank matrix completion problem by solving the SDP relaxation is guaranteed on average if the given matrix satisfies certain properties \cite{rip}. Recently, utilizing the SDP relaxation, Saunderson et al.~\cite{Saunderson_2012} established links between ($\Pc_2$) and two geometric problems: \emph{subspace realization} and \emph{ellipsoid fitting}, and provided a simple sufficient condition based on \emph{coherence of a subspace} for recovery of low-rank solutions. Beyond SDP relaxations, some other heuristics have also been proposed in the literature for such problems. In \cite{muzhu}, the authors interpret such a problem from the statistical perspective and suggest a blockwise coordinate descent algorithm. Their numerical experiments show that such heuristics perform well in terms of Kullback–Leibler loss. In \cite{adaptive}, the authors place such problems into the framework of  Riemannian optimization and
propose a so-called Riemannian rank-adaptive method, which involves fixed-rank optimization, rank increase step and rank reduction step.  In \cite{Jellyfish}, the authors propose a stochastic gradient algorithm which implements a projected incremental gradient
method with a biased, random ordering of the increments.  \textcolor{blue}{Bertsimas et al.\ \cite{bertsimas2017certifiably} consider $(\Pc_1)$, obtaining provable bounds that can be used in a branch-and-bound algorithm.}

\textcolor{blue}{There are also some other areas connected to the three optimization problems above. In theoretical computer science, \textit{matrix rigidity} is a measurement of a matrix's Hamming distance from any low rank matrix \cite{ramya2020recent}. More specifically, the rigidity of a matrix $A \in \SS^n$ for rank $r$ is minimum number of entry-wise changes to $A$ so that its rank becomes $r$. Denote this minimum number by $R_r(A)$. The decision version of $(\Pc_2)$ is that whether $A$ can be made rank $r$ by changing at most $n$ diagonal entries. In this case, the decision version of $(\Pc_2)$ can be viewed as a restricted rigidity problem. If $A$ has objective value $\Bar{r}$ in $(\Pc_2)$, then $R_{\Bar{r}}(A) \leq n$. Another related  research area is \textit{algebraic factor analysis} from algebraic statistics community \cite{drton2007algebraic}. Algebraic factor analysis considers algebraic structure of the family of multivariate normal distributions whose covariance matrix is a sum of a low-rank positive semidefinite matrix and a diagonal matrix. While the viewpoint of \cite{drton2007algebraic} is from the area of computational algebraic geometry, our focus is  on optimization and computational complexity in decomposing those covariance matrices.} 

There are various results establishing the hardness of computing low dimensional geometric representations of graphs, such as Saxe~\cite{Saxe1980} and Peeters~\cite{PeetersM.J.P1996ORoF}.
For the hardness of low-rank matrix completion, see Shitov~\cite{shitov2016hard} or Bertsimas et al.\ \cite{bertsimas2021mixed}. 
However, it seems that the computational complexity of none of the low-rank optimization problems ($\Pc_1$), ($\Pc_2$), ($\Pc_3$) is directly covered by existing results.
Our results on completeness for the first-order theory of the reals in Section~\ref{sec:existsR} are most closely related to Shitov's and use some of his techniques.

% All those state-of-art heuristics illustrate their effectiveness via certain numerical tests and can be applied to our problem. However, the complexity result is barely covered.
% The most similar paper is done by   They 
% However, the result in \cite{rip} and \cite{Saunderson_2012} can be easily violated with some worst case instances

% In the literature, one common convex optimization-based (approximation) method for $\Pc_2$ is \textit{minimum trace factor analysis} (MTFA) and can be expressed as following semi-definite program:
% \begin{mini*}|l|
% {L \in \SS^n,\dv \in \R^n} { \tr(L)  }{}{}\tag{MTFA}
% \addConstraint{ A = L + \Diag(\dv) }{}{}
% % \addConstraint{ L_{ij} = 0, \forall ij \notin E \text{ and } i \neq j}{}{}
% \addConstraint{ L \succeq 0}{}{}
% \end{mini*}

% In \cite{cite-key}, if MTFA is feasible, then it will have a unique minimizer.
% Furthermore, let $\mathcal{U}$ of $\R^n$ denote the column space of $U$ and $\mu(\mathcal{U}) := \max\limits_{i \in \{1,\dots,n\}} \{ \Project_{\mathcal{U}} \ev_i\} $ denote the coherence of $\mathcal{U}$. In \cite{Saunderson_2012}, if $\mu(\mathcal{U}) < 1/2$, then MTFA will solve $\Pc_2$ exactly. On the other hand, if $\mu(\mathcal{U}) > 1/2$, then MTFA may fail to solve $\Pc_2$ exactly.
% Notice it is well-known that $ \mu(\mathcal{U}) \geq r / n$. This means if the optimal value is too large, MTFA may not work, which is consistent with our result.

% \textcolor{red}{we need to write the formation below; no need to write it as a minimization problem}

\textcolor{blue}{The essential technical tools used herein are:
\begin{itemize}
\item 
In Section 2, we focus on the inverse of a particular $r\times r$ principal submatrix of the completion (matrix $V$ in \eqref{char_sys}), which leads to a system of $O(r^2)$ polynomial equations in $O(r^2)$ variables with maximum degree at most $r$.  This yields a polynomial-time algorithm to solve ($\Pc_1$) or ($\Pc_2$) by solving the underlying polynomial systems, since these polynomial systems become tractable (albeit with constants terms in the complexity bounds growing exponentially with $r$) when the optimal value $r$  of ($\Pc_1$) or ($\Pc_2$) is $O(1)$.
\item 
In Section 3, we use a suitable modification of a construction by Peeters \cite{PeetersM.J.P1996ORoF}, which reduces graph 3-colorability to a modified version of $(\Pc_3)$.  We use
 Schur complementation that allows indeterminate entries on the diagonal to induce indeterminate entries in off-diagonal positions, thus reducing $(\Pc_3)$ to  $(\Pc_1)$ and $(\Pc_2)$.  See, e.g., \eqref{eq:bdef} below, in which a diagonal perturbation of the $(1,1)$ block induces off-diagonal perturbations to $A$.  A second technique is replacing matrix entries by $3\times 3$ blocks in our reductions (see, e.g., \eqref{eq:3by3}) to enforce constraints on diagonal entries.
 \item In Section 4, our argument is based on a construction of 
Shitov \cite{shitov2016hard} showing that a certain matrix completion problem is 
complete in the first-order theory of the reals.  We augment this argument again with Schur complementation and $3\times 3$ blocks, e.g., Lemma~\ref{lem:uniquecompletion3by3} below.
\end{itemize}}

Next, we define a version of $(\Pc_1)$ allowing approximate solutions so that inaccuracies and/or uncertainties in data $A$ and inaccuracies arising from finite precision computations can be addressed. 

\underline{Problem $(\Pct_1)$:}  Fix a polynomial function $p(n)$. Given $A \in \SS^n$, $r \in \Z_+$, $\epsilon \in [0,1]$ such that there exists $\dv_0 \in \R^n$ and $H_0 \in \SS^n$ satisfying $\norm{H_0}_F \leq \epsilon$, $\dv_0 \geq \mathbf{0}$, $A - \Diag(\dv_0) + H_0 \succeq 0$, and $\rank(A - \Diag(\dv_0) + H_0) \leq r$, find $\dv \in \R^{n}_+$ and $H \in \SS^n$ such that $\norm{H}_F \leq p(n)\cdot \epsilon$, $A - \Diag(\dv) + H \succeq 0$  and  $\rank(A - \Diag(\dv) + H) \leq r.$

Problem $(\Pct_1)$ can be equivalently stated in terms of an optimization problem. Given $A \in \SS^n$ and a positive integer $r$, consider:
\begin{mini*}|l|
{R \in \R^{n \times r}, \dv \in \R^n} {\norm{A -R R^{\top} - \Diag(\dv)}_F}{}{}
\addConstraint{ \dv \geq \mathbf{0}.}{}{}
\end{mini*}
Then, problem $(\Pct_1)$ asks for a feasible solution $(\bar{R}, \bar{\dv})$ of the above optimization problem such that the objective function value of $(\bar{R}, \bar{\dv})$ is within a multiplicative factor $p(n)$ of the optimal objective value, for a polynomial $p$.

\textcolor{blue}{We allow a multiplicative factor of $p(n)$ as the approximation factor partly to ensure that our NP-hardness result is not due solely to the wrong choice of norm.  To give an example of the significance of the norm, consider the following problem.  Given $A\in\R^{n\times n}$, find
\[
\min\{\Vert A-A'\Vert_{\max}:A'\mbox{ is singular}\}.
\]
Here, $\Vert\cdot\Vert_{\max}$ denotes the maximum absolute entry of the matrix.  This problem is known to be NP-hard (see \cite{PoljakRohn1993} and the related results in \cite{Nemirovskii1993}).  On the other hand, the problem would be easily solved if it had been stated in terms of the operator-2-norm or Frobenius norm, since in this case the nearest singular matrix is found using the singular value decomposition.  By allowing the multiplicative factor of $p(n)$ in $(\Pct_1)$, we insulate the result against the possibility that the problem is polynomial-time solvable in some other norm.}

Moreover, we define ($\Pch_1$) which is similar to ($\Pct_1$) with an additional constraint on the sparsity pattern of $H$: $H_{ij} = 0$ if $A_{ij} = 0,\forall i \neq j$, and similarly for $H_0$. \textcolor{blue}{The rationale behind this format of ($\Pch_1$) is as follows. In many real life applications, some zero entries of the input data usually play important roles as prior information and they might be unchangeable. For example, when $A$ is an empirical covariance matrix, some zero entries of $A$ may refer to pairs of random variables which are not (linearly) correlated (therefore, these entries may not be altered).  Note that if $A$ is dense, then $(\Pch_1)$ and $(\Pct_1)$ are identical.}

% We would also point out that one could see in the later section that $\Pch_1$ behaves quite differently compared to $\Pct_1$. Our complexity argument for $\Pc_1$ applies to $\Pch_1$ most directly while complexity argument for $\Pc_1$

Likewise, we define ($\Pct_2$), ($\Pct_3$), ($\Pch_2$), ($\Pch_3$) in a similar way to the above.

% \begin{mini*}|l|
% {\dv \in \R^n, H \in \SS^{n}} { 0 }{}{}\tag{$\Pch_1$}
% \addConstraint{ \rank(A - \Diag(\dv) + H) \leq r}{}{}
% \addConstraint{ A - \Diag(\dv) + H   \succeq 0}{}{}
% \addConstraint{ \dv \geq 0}{}{}
% \addConstraint{ \norm{H}_F \leq 2 \epsilon}{}{}
% \addConstraint{ H_{ij} = 0, \forall i \neq j \text{ and } A_{ij} = 0  }{}{}
% \end{mini*}
% where $A \in \SS^n_+$ is given such that there exists $\dv \in \R^n, W \in \SS^{n}$ that $\rank(A-\Diag(\dv) + W) \leq  r$, $A-\Diag(\dv) + W \succeq 0$, $\Diag(\dv) \geq 0$ and $\norm{W}_F \leq \epsilon$. 
% Likewise, we define $\Pch_2,\Pch_3$ in the similar way as above.

Our results on the computational complexity of these low-rank decomposition problems will consider
rational data and the Turing machine model as well as real number data and the Blum-Shub-Smale real number machine model \cite{Blum1989OnAT}.
In this paper, we will show following results in order:

\begin{enumerate}[(i)]
    \item \textcolor{blue}{In Section 2,
    Theorem~\ref{thm_given}  shows when the optimal value of ($\Pc_2$) is $O(1)$, ($\Pc_2$) can be solved in polynomial time. The same statement holds for ($\Pc_1$) (Theorem~\ref{thm_given_2}).}
%    \item ($\Pc_3$) can be reduced to ($\Pc_2$) in polynomial time; ($\Pc_2$) can be reduced to ($\Pc_1$) in polynomial time under a mild assumption.
    \item \textcolor{blue}{Theorem~\ref{thm:tilde_p3_is_nphard} shows that $(\Pc_3)$, $(\Pct_3)$, and $(\Pch_3)$ are NP-hard. Theorem~\ref{thm_P1_tilda} shows that  $(\Pc_1)$, $(\Pct_1)$, and $(\Pch_1)$ are NP-hard.  Theorem~\ref{thm_P2} shows that $(\Pc_2)$ and $(\Pch_2)$ are NP-hard, while Theorem~\ref{thm:p2til_nphard} in the Appendix shows the result for $(\Pct_2)$.}
      \item \textcolor{blue}{Theorem~\ref{thm:P3redP2} reduces $(\Pc_3)$ to $(\Pc_2)$, and Theorem~\ref{thm:red_poly_p3} reduces solution of a system of polynomial equations to $(\Pc_3)$.  Thus, the results of Section 4 show both $(\Pc_2)$ and $(\Pc_3)$ are complete for the first-order theory of the reals.}
\end{enumerate}

\end{section}

%  \textcolor{red}{fix all r to $r_A$ or some}
 
%   \textcolor{red}{check diag Diag}
  
%   \textcolor{red}{check all numbers}

%   \textcolor{red}{in polynomial time}

%  \textcolor{red}{The variable list (delete it later)}
% \begin{enumerate}
%     \item $A_1$: simple example
%     \item $U$
% \end{enumerate}

\begin{section}{Algorithms}
\label{sec:alg}

In this section, we prove that ($\Pc_1$) and ($\Pc_2$) can be solved in polynomial time if their optimal objective value is bounded above by an
absolute constant. 
 We begin with representing the set of optimal solutions to ($\Pc_2$) as the solution set of a polynomial system of equations. Then, we propose an algorithm to solve ($\Pc_2$) which runs in polynomial time if the optimal value is $r = O(1)$. Such an algorithm can be extended to handle ($\Pc_1$) as well. In a later section, we will show that ($\Pc_3$) can be reduced to ($\Pc_2$) in polynomial time. However, unlike ($\Pc_1$) and ($\Pc_2$), $(\Pc_3$) may be hard even when the optimal objective value (rank) is $O(1)$. Indeed, our reduction from ($\Pc_3$) to ($\Pc_2$) increases the optimal objective value (rank) substantially. We continue with some well-known facts about Schur complements in relation to positive semidefiniteness and rank. \textcolor{blue}{For similar results and their proofs, see, for instance \cite{zhang2006schur}.}

\begin{lem} 
Assume $A \in \SS^n$, $W \in \R^{k \times n}$ and for some $B\in \SS^k$,
the block matrix
$$M := \begin{bmatrix}
   A & W^{\top} \\ 
    W & B
\end{bmatrix}$$
is positive semidefinite. Then, there exist $V_1, V_2 \in \R^{k \times n}$ such that $W = V_1 A = BV_2$ .
\label{lem:rangeA1}
\end{lem}

\begin{lem}
\label{lemma_schur}
\begin{enumerate} 
\item 
(Symmetric positive semidefinite case.)
For $A \in \SS_+^n$ and $W \in \R^{k \times n}$, suppose the columns of $W^{\top}$ lie in $\Range(A)$, and in particular, say $W=VA$ for some $V \in \R^{k \times n}$. Then
$$M := \begin{bmatrix}
   A & W^{\top} \\ 
    W & B
\end{bmatrix} \succeq 0$$ if and only if $B \succeq V A V^{\top}$.
Moreover, $$\rank \left(M\right) = \rank\left(A\right) + \rank(B - V A V^{\top}).$$  In the special case that $A$ is nonsingular, this formula simplifies to $$\rank(M) = n + \rank(B - V A V^{\top}) = n+\rank(B-W A^{-1}W^{\top}) .$$

\item (Unsymmetric case.)
Under the assumption that $A\in\R^{n\times n}$ is nonsingular in the block matrix
$$M := \left(\begin{array}{cc}
A & B \\
C & D
\end{array}\right),
$$
where $B,C,D$ have conforming sizes, $\rank(M)=n+\rank(D-CA^{-1}B)$.
\end{enumerate}
\end{lem}

%\begin{proof}
%For Part (1), suppose $A \succeq 0$, $B \in \SS^{k},$ and $W = AV$. Then
%$$\begin{bmatrix}
%    I & 0 \\
%    -V^{\top} & I 
%\end{bmatrix} \begin{bmatrix}
%    A & W \\
%    W^{\top} & B
%\end{bmatrix} \begin{bmatrix}
%    I & -V \\
%    0 & I
%\end{bmatrix} = \begin{bmatrix}
%    A & 0 \\
%    0 & B - V^{\top} A V
%\end{bmatrix}.$$

%Let $S := \begin{bmatrix}
%     I & 0 \\
%    -V^{\top} & I 
%\end{bmatrix}$. One can check that the linear map $f(M) : \SS^{n+k} \to \SS^{n+k} :=  S M S^{\top}$ preserves the rank, the symmetry and the positive %semidefiniteness of its argument since $S$ is square, non-singular ($\det(S) = 1$). 

%The proof of Part (2) uses the analogous reduction to the block-diagonal case.
%\end{proof}

The following lemma describes the set of optimal solutions of ($\Pc_2$). To set up the next lemma and its proof, note that for any $A \in \SS^n$ with zero diagonal, $\dv := -\lambda_n(A) \ev$
gives a feasible solution of ($\Pc_2$) with objective value at most $n-1$. Next, notice that ($\Pc_2$) has a feasible solution with objective value $r$ iff there exists $U \in \R^{n \times r}$ such that $\rank(U) = r$ and $A +\Diag(\dv) = U U^{\top}$. Indeed, $\rank(U) = r$ iff there exists $J \subseteq \{1,2, \ldots,n\}$ such that $|J|=r$ and the matrix $U_J:= U(J,:)$ is nonsingular. Suppose such $J$ exists. We may assume $J=\{1,2, \ldots, r\}$, and let $\bar{J} := \{1,2, \ldots,n\} \setminus J$. Then we have the block matrix equation
\[
A +\Diag(\dv) = \begin{bmatrix}
    U_J U_J^{\top} & U_J U_{\bar{J}}^{\top} \\
    U_{\bar{J}} U_J^{\top}  &  U_{\bar{J}} U_{\bar{J}}^{\top}
\end{bmatrix}
\]
with (1,1) block symmetric positive definite. The following lemma exploits this structure. To obtain the system in the statement of the lemma, one can directly analyze the above matrix equation or utilize the properties of the Schur complements outlined in the previous two lemmas.

% In this section, we show that $\Pc_2$ can be solved in polynomial time under the assumption that the optimal value of $\Pc_2$ is $O(1)$ and we will propose an algorithm to solve $\Pc_2$ which runs in polynomial time if $r = O(1)$.
% Even though $\Pc_2$ is relatively simple, we will show later that $\Pc_3$ can be reduced to $\Pc_2$ in polynomial time and our  algorithm for $\Pc_2$ can be extended to handle $\Pc_1$ as well. As a result, we show that $\Pc_1,\Pc_2$ and $\Pc_3$ can be solved in polynomial time if $r = O(1)$.

% Let $A \in \SS^n$ that $\diag(A) = 0$ be the given input of $\Pc_2$. Suppose the optimal value of $\Pc_2$ is $r$, particularly, there exist $U \in \R^{n \times r} \text{ and } \dv \in \R^n$ such that $A = U {U}^{\top} + \Diag(\dv)$.

\begin{lem}
\label{lem_char_sys}
Let $n \geq 2$ be an integer, $A \in \SS^n$ with $\diag(A) = \mathbf{0}$ and $r \in \{1,2, \ldots,n-1\}$ be given. Then, $\dv \in \R^n$ is a feasible solution of $(\Pc_2)$ with objective function value $r$ if and only if there exists  $J \subseteq \{1,2, \ldots, n\}$ such that $|J| = r$, and with $\Bar{J} := \{1,2,\dots,n\} \setminus J$ the following system has a solution $(\dv,V) \in \R^n \times \SS^{r}:$
\begin{subequations}
\label{char_sys}
\begin{align*}
    & \left[A(J,i) \otimes A(J, j)\right]^{\top} \vectize(V) = A_{ij}, \forall i,j \in \Bar{J}, i < j ; \\ \tag{\ref{char_sys}} 
    & \left[A(J,i) \otimes A(J,i)\right]^{\top} \vectize(V) = d_i, i \in \Bar{J}; \\
    & \ev_i^{\top} V^{-1} \ev_j = A_{ij}, \forall i,j \in J, i < j ;\\
    & (V^{-1})_{ii} = d_i, i \in J; \\
    & V \in \SS_{++}^{r}.
\end{align*}

\end{subequations}
\end{lem}

\begin{proof}
Let $\dv \in \R^n$ be a feasible solution of ($\Pc_2$) with objective value $r \in \{1,2, \ldots,n-1\}$. Then, there exists $U \in \R^{n \times r}$ such that $U U^{\top} - \Diag(\dv) = A$ and $\rank(U) = r$. Let $J \subset \{1,2,\dots,n\}$ be such that $\rank(U_J) = |J| = r$. Further, let $V := (U_J U_J^{\top})^{-1} \in \SS_{++}^{r}$. Then, $(\dv,V)$ satisfies all the constraints in $(\ref{char_sys})$.

Conversely, for some $J \subset \{1,2,\dots,n\}$ with $|J| = r$, suppose $(\dv,V)$ solves $(\ref{char_sys})$. Compute $\widehat{U} \in \R^{r \times r}$ via the decomposition $\widehat{U} \widehat{U}^{\top} = V^{-1}$. Let 
\[
U(i,:) := \begin{cases} \widehat{U}(i,:), \text{ if } i \in J \\
A(i,J)\widehat{U}^{-T}, \text{ if } i \in \Bar{J}
\end{cases}
\]
where $\Bar{J} := \{1,2,\dots,n\} \setminus J$.
Then $[U U^{\top} - \Diag(\dv)]_{ij} = \begin{cases} 0, \text{ if } i = j; \\ A_{ij}, \text{ if } i \neq j. \end{cases}$

Therefore, $\dv$ is a feasible solution to ($\Pc_2$) with objective value $r$.
\end{proof}

% Since the polynomial system in Lemma \ref{lem_char_sys} is a complete characterization of optimal solution of $\Pc_2$ with a given index set. One key observation here is that the system (\ref{char_sys}) can be viewed as a polynomial system on $V \in \SS^r$ and such observation leads to the following algorithms

Since system (\ref{char_sys}) gives a complete characterization of all solutions $(\dv,V)$ corresponding to the given $J$, we can design an algorithm based on a systematic way of solving this system. During the first phase, we solve a linear system of equations and exploit the fact that we can check positive definiteness of a given symmetric matrix efficiently and in a numerically stable way. The next algorithm takes an index set $J$ as part of its input and tries to solve system \eqref{char_sys} using only those constraints that are easy to handle.

\SetArgSty{textnormal}

\begin{algorithm}
\label{alg1}
\SetAlgoLined
\SetKwInOut{Input}{input}\SetKwInOut{Output}{output}
% \KwResult{Write here the result }
\Input{$A \in \SS^n, \diag(A)=\mathbf{0}, J \subseteq \{1,2, \ldots, n\}.$ }
\Output{\textcolor{blue}{A feasible solution $\dv$ to $\Pc_2$ with objective value $|J|$ or an infeasibility certificate or a certificate that there may be a solution to the original problem but $U(J,:)$ is singular in any such solution or valid linear equations for feasible $V$ (which can be used in Algorithm 2).} }

$\Bar{J} := \{1,2, \ldots, n \} \backslash J, r := |J| $, solve:
\begin{equation}
    \label{inner}
    [A(J,i) \otimes A(J,j)]^{\top} \vectize(V) = A_{ij}, \forall i < j \text{ and } i,j \in \Bar{J}
\end{equation}

\uIf{(\ref{inner}) has no solution}{\Return an infeasibility \text{certificate} that either the original problem is infeasible or there may be a solution to the original problem but $U(J,:)$ is singular in any such solution.} 
\uElseIf{(\ref{inner}) has infinitely many solutions}{\Return $B \in \R^{k \times r(r+1)/2},\bv \in \R^{k}$ such that $\rank(B) = k$, $V$ solves (\ref{inner}) if and only if $B \svec(V) = \bv$ }
\ElseIf{the system (\ref{inner}) has a unique solution $V$ }{
\uIf{$V \notin \SS_{++}^r$ \textbf{ or }  $\ev_i^{\top} V^{-1} \ev_j \neq A_{ij}, \exists i < j \text{ and } i,j \in J$}{\Return the corresponding certificate that either the original problem is infeasible or there may be a solution to the original problem but $U(J,:)$ is singular in any such solution.}
\Else{compute $\wv \in \R^n$  by: $w_i := \begin{cases} (V^{-1})_{ii}, \forall i \in J \\ A(J,i)^{\top} V A(J,i), \forall i \in \Bar{J} \end{cases}$ \\
\Return $\dv:= \wv$
}
}

% If (\ref{inner}) has no solution, \textbf{STOP} with an infeasibility \text{certificate}. 

% If (\ref{inner}) has infinitely many solutions, return $B \in \R^{k \times r(r+1)/2},b \in \R^{k},\rank(B) = k$ such that $V$ solves (\ref{inner}) if and only if $B \svec(V) = b$ and \textbf{STOP} .

% If the system (\ref{inner}) has a unique solution $V$:

% - Check that $V \in \SS_{++}^r$. If not, \textbf{STOP} with a certificate.

% - Check that $\ev_i^{\top} V^{-1} e_j = A_{ij}, \forall i < j \text{ and } i,j \in J$. If not, \textbf{STOP} with a certificate

% Diagonal entries of $R$ are given by: $R_{ii} := \begin{cases} (V^{-1})_{ii}, \forall i \in J \\ A(J,i)^{\top} V A(J,i), \forall i \in \Bar{J} \end{cases}$

% Return $d:= \diag(A) - \diag(R)$, \textbf{STOP}

%  \For{all possible $r$-by-$r$ submatrix U}{
%     // {assume U is a leading submatrix after permutation} \\
%     Try: \\
%     \ \ \ \ Establish a linear system in \cref{eq:2} and solve $V$ \\
%     \ \ \ \ Compute $U^{-1}$ such that $V = U^{-\top} U^{-1}$ via Cholesky decomposition \\ 
%     \ \ \ \ Compute $U$ and verify \cref{eq:3} \\
%     \ \ \ \ If none of above procedures fails, return $U$. Otherwise continue
%  }

 \caption{linear solver with a given index set}
\end{algorithm}

If Algorithm~\ref{alg1} fails to solve the problem ($\Pc_2$) and returns a linear system of equations whose solution set contains all solutions to ($\Pc_2$) for the given index set $J$, we make use of this information and proceed with the second phase in which we solve a system of polynomial equations.

\begin{algorithm}
\label{alg2}
\SetAlgoLined
\SetKwInOut{Input}{input}\SetKwInOut{Output}{output}
% \KwResult{Write here the result }
\Input{$A \in \SS^n, \diag(A)=\mathbf{0}, J := \{j_1,j_2,\dots,j_r \} \subseteq \{1,2,...,n\}, B \in \R^{k \times \frac{r(r+1)}{2}},\bv \in \R^k$ that $\rank(B)= k$}
\Output{\textcolor{blue}{Feasible solution $\dv$ of $\Pc_2$ with objective value of $k$ or an infeasibility certificate.} }

solve:
\begin{equation}
    \label{inner2}
    \left\{
    \begin{aligned}
    & B \svectize(V) = \bv   \\ 
    & A_{ij} \det(V) - \adj(V(i,j)) = 0, \forall i < j , i,j \in J \\
    & \det(V_{J_kJ_k}) z_k^2 = 1 ,\forall k \in \{1,2,\dots,r\} \text{ where } J_k := \{ j_1,j_2,\dots,j_k\} \\
    & V \in \SS^r, \zv \in \R^r
    \end{aligned}
    \right.
\end{equation}

\uIf{ (\ref{inner2}) does not  have a solution  }{\Return an infeasibility certificate}
\Else{let $(V,\zv)$ denote a solution of (\ref{inner2}) \\ compute $\wv \in \R^n$  by: $\wv_{i} := \begin{cases} (V^{-1})_{ii}, \forall i \in J \\ A(J,i)^{\top} V A(J,i), \forall i \in \Bar{J} \end{cases}$ \\
\Return $\dv:= \wv$ }

% If \cref{inner2} does not  have a solution, \textbf{STOP} with a certificate. Otherwise, let $V$ be any solution of \cref{inner2}. 

% Diagonal entries of $R$ are given by: $R_{ii} := \begin{cases} (V^{-1})_{ii}, \forall i \in J \\ A(J,i)^{\top} V A(J,i), \forall i \in \Bar{J} \end{cases}$

% Return $d := \diag(A) - \diag(R)$, \textbf{STOP} 

% Find an index set $S \subset \{1,2,..,n\}: |S| > r $ \\ 
%  \For{all possible $r$-by-$r$ submatrix U from $S$}{
%     // {assume U is a leading submatrix after permutation} \\
%     Try: \\
%     \ \ \ \ Establish a linear system in \cref{eq:2} and solve $V$ \\
%     \ \ \ \ Compute $U^{-1}$ such that $V = U^{-\top} U^{-1}$ via Cholesky decomposition \\ 
%     \ \ \ \ Compute $U$ and verify \cref{eq:3} \\
%     \ \ \ \ If none of above procedures fails, return $U$. Otherwise continue
%  }
 
 \caption{nonlinear solver with a given index set}
\end{algorithm}

In Algorithm~\ref{alg2},  $B \svec(V) = \bv$ may be written as $[A(J,i) \overset{s}{\otimes} A(J,j)] \svectize (V) = A_{ij}$ where $\overset{s}{\otimes}: \R^r \times \R^r \to \R^{r(r+1)/2}$ is symmetric Kronecker product and $\svectize: \SS^r \to \R^{r (r+1)/2}$ returns the lower triangular part of a symmetric matrix. Also $A_{ij} \det(V) - \adj(V(i,j)) = 0$ can be written as $\ev_i^{\top} V^{-1} \ev_j = A_{ij}$, provided $\det(V) \neq 0.$

Since Algorithm~\ref{alg1} can be implemented very efficiently for large scale instances, in some applications it might be worthwhile run Algorithm~\ref{alg1} for many different subsets $J$ before resorting to Algorithm~\ref{alg2} (as one might get lucky). \textcolor{blue}{The input of Algorithm~\ref{alg1} is the original data and a prescribed index set and the algorithm can be run for different index sets independently. Therefore, the algorithm is easily parallelizable.} 
However, in this approach, in the worst case, we might have to consider every subset $J$ (of cardinality $r$) of $\{1,2, \ldots,n\}$ to determine whether $(\Pc_2$) has a feasible solution with objective value $r$:

\begin{algorithm}
\label{alg3}
\SetAlgoLined
\SetKwInOut{Input}{input}\SetKwInOut{Output}{output}
% \KwResult{Write here the result }
\Input{ $A \in \SS^n, r \in \{1,2, \ldots, n-1\}$ }
\Output{\textcolor{blue}{Feasible solution of $\Pc_2$ with objective value of $r$ or an infeasibility certificate.}}
\For{all possible $J \subseteq \{1,2, \ldots,n\}$ such that $|J| = r$  }{

Run \textbf{Algorithm 1} with input $A$ and $J$. If \textbf{Algorithm 1} succeeds, then return its output. Otherwise, if \textbf{Algorithm 1} fails and returns a linear system of equations, run \textbf{Algorithm 2} with input $A$ and $J$ and the returned linear system. If \textbf{Algorithm 2} succeeds, then return its output.
% call \textbf{Algorithm 1} with input $A$ and $J$ \\
% \uIf{\textbf{Algorithm 1} returns $\dv \in \R^n$}{\Return $\dv$}
% \uElseIf{\textbf{Algorithm 1} returns $B \in \R^{k \times \frac{r(r+1)}{2}},\bv \in \R^k $}{choose $B,\bv$ as the same as output of \textbf{Algorithm 1}}
% \Else{\textbf{continue}}

% call \textbf{Algorithm 2} with input $A,J,B,\text{ and }\bv$ \\
% \If{\textbf{Algorithm 2} returns $\dv \in \R^n$}{\Return $\dv$}

}

 \caption{solver without a given index set}
\end{algorithm}

\begin{example}
Let $$A := \begin{bmatrix}
    0 & 1 & 2 & 1 & 0 \\
    1 & 0 & 2 & 0 & 1 \\
    2 & 2 & 0 & 0 & 0 \\
    1 & 0 & 0 & 0 & 1 \\
    0 & 1 & 0 & 1 & 0 \\
\end{bmatrix},$$
and consider the corresponding instance of ($\Pc_2$). The 3-by-3 submatrix of $A+\Diag(\dv)$ identified by the rows
$\{1,2,3\}$ and the columns $\{1,4,5\}$ is $\begin{bmatrix}
    d_1 & 1 & 0 \\
    1 & 0 & 1 \\
    2 & 0 & 0 \\
\end{bmatrix}$. Since this matrix is nonsingular for every $\dv \in \R^5$, $\rank(A+\Diag(\dv)) \geq 3$ for every $\dv \in \R^5$. Hence, the optimal objective value of ($\Pc_2$) is at least three. Next, consider $\dv := [2,2,3,2,2]^{\top}$. Then, $A + \Diag(\dv) \succeq 0 $ and $\rank(A + \Diag(\dv)) = 3$. So, $\dv$ is an optimal solution. If we call Algorithm 1 with inputs $A$ and $J := \{1,2,3\}$, the system (\ref{inner}) will have infinitely many solutions described by
\[
V^{-1} = \begin{bmatrix}
    \alpha & 1 & 2 \\
    1 & \beta & 2 \\
    2 & 2 & \gamma
\end{bmatrix} \textup{ and } V_{21} = 1,
\]
where $\alpha,\beta,\gamma \in \R$. So, Algorithm 1 will fail to completely solve ($\Pc_2$) with input $A$ and $J$. Therefore, for this instance, Algorithm 2 is needed to solve ($\Pc_2$). The solution set of the resulting problem is characterized by two parameters $\alpha$, $\beta$ (and there is a unique value for $\gamma$, $\gamma := 4(\alpha+\beta -1)/(\alpha \beta)$):
\[
A (\alpha,\beta) := \begin{bmatrix}
    \alpha & 1 & 2 & 1 & 0 \\
    1 & \beta & 2 & 0 & 1 \\
    2 & 2 & \frac{4(\alpha+\beta -1)}{\alpha \beta} & 0 & 0 \\
    1 & 0 & 0 & \frac{\beta}{\alpha -1} & 1 \\
    0 & 1 & 0 & 1 & \frac{\alpha}{\beta -1} \\
\end{bmatrix},\]
where $\alpha >0$, $\alpha \beta > \max\{1, \alpha+\beta -1\}$. Thus, Algorithm 2 will return the diagonal of one of the above given infinitely many matrices.

Let us modify this example so that for our subset $J$, there are exactly two optimal solutions. We can require that the range of $(A+\Diag(\dv))$ contain the vector
$[1, 1, -1, 5, 5]^{\top}$ for every feasible solution $\dv$ for the
original matrix $A$, by redefining $A$ as
\[
A := \begin{bmatrix}
    0 & 1 & 2 & 1 & 0 & 1\\
    1 & 0 & 2 & 0 & 1 & 1\\
    2 & 2 & 0 & 0 & 0 & -1\\
    1 & 0 & 0 & 0 & 1 & 5\\
    0 & 1 & 0 & 1 & 0 & 5\\
    1 & 1 & -1 & 5 & 5 & 0\\
\end{bmatrix}.
\]
Then, the optimal objective value is still equal to three, and all optimal solutions are characterized by $\alpha \in \{2,4\}$ and $\beta:=\alpha$ (and for each of the choices, the unique entry for $d_6$ that does not increase the rank).
%This instance of ($\Pc_2$) has two optimal solutions and Algorithm 2 finds that
%$(\alpha, \beta, \gamma) \in \left\{(2,2,3), (4,4,7/4)\right\}$ describes the solution %set of the nonlinear system \eqref{inner2}.
\end{example}

The above example also shows that there are instances of ($\Pc_2$) with exactly two optimal solutions. This observation exposes possible complexity of optimal solution sets of ($\Pc_2$) and hints at the existence of possible gadgets for NP-hardness proofs.

Since Algorithm~\ref{alg3} can determine whether a given instance of ($\Pc_2$) has a feasible solution with objective value $r$, if the optimal objective value is $\bar{r}=O(1)$, we can solve the problem instance by enumerating all possible values for $r \in \{1,2, \ldots, \bar{r}\}$. 

% \textcolor{red}{rewrite the whole section}

% \textcolor{red}{Cylindrical algebraic decomposition}

% \textcolor{red}{P1, P2 is theorem}

\begin{thm}
\label{thm_given}
\label{rem_ukn_small}
If $r = O(1)$, then Algorithm~\ref{alg3} can be implemented in polynomial time. 
Thus, if the optimal value of an instance of ($\Pc_2$) is $\bar{r}=O(1)$, then such instances of ($\Pc_2$) can be solved in polynomial time by enumerating all possible ranks from $1$ to $\bar{r}$ and by calling Algorithm~\ref{alg3} for each possible rank.
\end{thm}

\begin{proof}
Algorithm~\ref{alg1} can be implemented to run in polynomial time, since it involves solving a linear system of equations whose size is bounded by a polynomial function of the size of the input $A$ and performing a Cholesky decomposition (e.g., $LDL^{\top}$) on an $r$-by-$r$ matrix. If $r=O(1)$, since the system (\ref{inner2}) is a system of polynomial equations with $O(r^2)=O(1)$ variables and $O(r^2)=O(1)$ equations, Algorithm~\ref{alg2} can be implemented to run in $O(1)$ time (e.g., by cylindrical algebraic decomposition, see for instance, \cite{Collins1975,RRE2000,BPR2006,BD2007,BB2012} and the references therein). In enumerating all possible ranks $\{1,2, \ldots, \bar{r}\}$, Algorithm~\ref{alg3} calls Algorithm~\ref{alg1} and Algorithm~\ref{alg2} at most $$\sum_{r=1}^{\bar{r}} \choose{n}{r}=O\left(n^{O(1)}\right)$$ times. Therefore, if $\bar{r}=O(1)$, we can solve ($\Pc_2$) in polynomial time.
\end{proof}

% \textcolor{red}{combine the 2 lemma}

% \begin{rem}
% \label{rem_ukn_small}
% If $r$ is $O(1)$,
% \end{rem}

The above approach can be extended to solving $(\Pc_1$).

\begin{thm} 
\label{thm_given_2}
Every instance of ($\Pc_1$) with optimal value $\bar{r}=O(1)$ can be solved in polynomial time.
\end{thm}

\begin{rem}
It should be noted that the obvious approach for extending the preceding algorithm and our proof to the case of $(\Pc_1)$, namely, introducing $d_1,\ldots,d_n$ as variables and constraining them either via inequalities, i.e., $d_j\ge 0$ for $j=1,\ldots,n$, or via equalities, i.e., $d_j=y_j^2$ for  $j=1,\ldots,n$, where $y_j$'s are new real variables and repeating the same arguments on this modified system, will not establish this theorem.  This is because the introduction of $O(n)$ new variables into the polynomial system will lead to an algorithm exponential in $n$.
Therefore, the following proof develops an algorithm that uses only $O(r^2)$ variables.
\end{rem}

\begin{proof}
First, we make modify \eqref{char_sys} to account for the fact that in $(\Pc_1)$, the diagonal entries of the given $A$ are positive and that $\dv$ is subtracted from
rather than added to diagonal entries to obtain the following system.
\begin{equation}
\begin{aligned}
    & \left[A(J,i) \otimes A(J, j)\right]^{\top} \vectize(V) = A_{ij}, \forall i,j \in \Bar{J}, i < j ; \\ 
    & \left[A(J,i) \otimes A(J,i)\right]^{\top} \vectize(V) = A(i,i)-d_i, i \in \Bar{J}; \\
    & \ev_i^{\top} V^{-1} \ev_j = A_{ij}, \forall i,j \in J, i < j ;\\
    & (V^{-1})_{ii} = A(i,i)-d_i, i \in J; \\
    & V \in \SS_{++}^{r}; \\
    & d_i\ge 0, \forall i=1,\ldots,n.
\end{aligned}
\label{char_sys2}
\end{equation}
Next, we treat $d_i$ for $i\in J$ as new system variables, and constrain them via $d_i\ge 0$ (or, equivalently, $d_i=y_i^2$ as noted above). This introduces an additional $O(r)$ variables into the system. We modify Algorithm~\ref{alg1} and Algorithm~\ref{alg2} accordingly.

We drop the second constraint of \eqref{char_sys2}, and we narrow the scope of the last constraint to: $d_i\ge 0$ for all $i\in J$. Next, treat $d_i$ for each $i\in \bar{J}$ as a function of the $O(r^2)$ variables:
\begin{equation}
d_i(V)=A(i,i) - \left[A(J,i) \otimes A(J,i)\right]^{\top} \vectize(V).
\label{eq:diV}
\end{equation}
Whenever Algorithm~\ref{alg1} terminates with a unique solution, we simply check whether $d_i(V) \geq0$ for $i\in\bar{J}$. If so, we stop, we found a desired solution (for the $r$ value being tested). Otherwise, either Algorithm~\ref{alg1} correctly identifies the fact that ($\Pc_1$) has no feasible solution with the tested objective value $r$ where the tested subset $J$ gives a nonsingular $U_J$, or Algorithm~\ref{alg1} returns a linear system of equations which hold at every solution of ($\Pc_1$) corresponding to the given $J$. 

For Algorithm~\ref{alg2}, to solve the system of polynomial equations in $(V,\zv,\dv(J))$, we choose an implementation (e.g., one based on cylindrical algebraic decomposition) that returns a point from each connected component of the solution set of the polynomial system of equations (\ref{inner2}). Since $r=O(1)$, the number of connected components of the solution set of (\ref{inner2}) is $O(1)$. If one of these returned points corresponds to a nonnegative $\dv(\bar J)$, we are done. Otherwise, each connected component either does not contain any points corresponding to a nonnegative $\dv(\bar J)$, or its corresponding points in the $\dv(\bar{J})$-space crosses at least one of the facets of the nonnegative orthant in dimension $n-r$, and such a crossing point is also a valid solution.   So, %for each of these connected components
to find these crossings,
we add an equation corresponding to each facet of the nonnegative orthant and solve the new polynomial system with one additional equation in the original variables. This equation has the form $d_i(V)=0$ for some $i\in \bar{J}$, where $d_i(V)$ is the affine linear function of $V$ given by \eqref{eq:diV}.

Our argument about the status of the connected components above applies to each facet of the nonnegative orthant. Thus, we apply the procedure recursively.  
Since $(V,\zv,\dv(J))$ lies in $\SS^{r} \times \R^r\times \R^r$, the system is solved in $O(1)$ time.  At first glance it seems that the recursion depth is $O(n)$ since the recursion may need to eventually include every equation $d_i(V)=0$ for each $i\in\bar{J}$.  If the recursion proceeds to depth $n-r$, the algorithm is still exponential in $n$.

However, we can make the following observation.  Each added equation of the form $d_i(V)=0$ for some $i\in\bar{J}$ is a linear equation.  The %\sout{most} 
number of non-redundant linear equations that can be added to a system of linear and nonlinear equations in $O(r^2)$ variables is \textcolor{blue}{at most} $O(r^2)$. %\sout{equations}.  
Therefore, the recursion must stop after at most $O(r^2)$ levels.
%When adding a linear equation to an existing system of linear equations (plus additional nonlinear equations), there are three possible outcomes.  The first possible outcome is that the new equation is inconsistent with the existing linear equations, in which case the recursion stops because the system overall is infeasible.  The second possible outcome is that the new equation is redundant with the existing linear equations, in which case no new recursion is needed.  The last possible outcome is that the new equation lowers the dimension of the affine manifold defined by the linear equations by 1.  In the third case, since $V$ lies in a space of dimension $\frac{r(r+1)}{2}$, the most number of such non-redundant equations that can be added is $\frac{r(r+1)}{2}$.  This bounds the depth of the recursion to a polynomial in $r$, which we interpret as $O(1)$, and 
Thus, the overall running time is $O(n^{r^2})$, where the leading coefficient is exponential in $r$.
\end{proof}

% Although \textbf{Algorithm 3} has a good theoretical complexity, it may not be very practical compared to MTFA as it may take too many iterations and may need to solve a big real polynomial equations. However, \textbf{Algorithm 1} is easier than MTFA because it only involves elementary linear operations. One necessary condition for\textbf{Algorithm 1} to succeed is that $U_J$ is non-singular and this condition is different from coherence condition of MTFA. Consider $U_0^{\top} \in \R^{r \times n}$ that $\mu(U_0) < 1/2$ and construct $U_1^{\top} \in \R^{r \times (n + k) } := [U_0 ; \textbf{0}]$ and $A_1 = U_1 U_1^{\top} + \Diag(\dv)$ for some $\dv \in \R^{n+k}$. Since $\mu(U_0) = \mu(U_1)$, MTFA will still compute exact solution for $\Pc_2$ while \textbf{Algorithm 3} will take considerable number of iterations because ${U_1}_J$ is singular with high probability for a randomly picked index set. 

\textcolor{blue}{Both Theorem \ref{thm_given} and \ref{thm_given_2} indicate that ($\Pc_1$), ($\Pc_2$) are fixed-parameter tractable with respect to the parameter $\bar{r}$. The complexity of ($\Pc_1$) and ($\Pc_2$) depends on the implementation of Algorithm 2. If Algorithm 2 is solved by cylindrical algebraic decomposition, it can be done in $O(n^{2^{O(\bar{r}^2)}})$ time \cite{Collins1975}. Since we run Algorithm 2 up to $O(\bar{r} n^{\bar{r}})$ times and $\Bar{r} \leq n$, the total complexity is $O(\bar{r} n^{\bar{r}} n^{2^{O(\bar{r}^2)}}) = O(n^{2^{O(\bar{r}^2)}})$.}  
Beyond some special cases, such as when the optimal objective value of ($\Pc_1$), ($\Pc_2$) can be bounded by an absolute constant, we do not know how to construct polynomial time algorithms for these problems in general. Indeed, as we show in the next section, these problems are NP-hard in general.

\end{section}

\begin{section}{NP-hardness.}

\label{sec:np-hard}
In this section, we prove NP-hardness of problems $(\Pc_1)$, $(\Pch_1)$,  $(\Pct_1)$, $(\Pc_2)$, $(\Pch_2)$,  $(\Pc_3)$,
$(\Pch_3)$, and $(\Pct_3)$.  The proof of NP-hardness of $(\Pct_2)$ is deferred to an appendix for reasons explained below.

The reductions are all from 3-coloring, and all rely on a construction by Peeters \cite{PeetersM.J.P1996ORoF} described below.  We will provide the reductions for each problem and prove their correctness.  We will not provide explicit proofs that each construction is polynomial-time for a Turing machine since it should be apparent to the reader that each reduction involves writing down numbers whose length is polynomial in the size of the original input graph.

The following preliminary lemmas are used to establish the correctness of the reductions.

\begin{lem}
\label{lem_bound_of_inverse_of_sum}
Let $A,B \in \R^{n \times n}$ such that $A^{-1}$ exists and $\norm{ B}_F\cdot\norm{ A^{-1} }_F < 1$. Then $A+B$ is invertible and
$(A+B)^{-1} =  A^{-1} + Z$ where $$\norm{Z}_F \leq \norm{B}_F (\norm{A^{-1}}_F )^2/\left( 1 - \norm{B}_F \norm{A^{-1}}_F\right).$$
\end{lem}

\begin{proof}
We claim that the formula for $(A+B)^{-1}$ is
$$(A+B)^{-1}=A^{-1}  + A^{-1}  \sum_{i=1}^{\infty} (-BA^{-1})^{i} .$$
First, note that the sum on the right-hand side is convergent since, by submultiplicativity, the $i$th term in the summation is bounded by $\Vert BA^{-1}\Vert_F^i$, i.e., a decreasing geometric series with a convergent sum.  Therefore, the right-hand side is well defined.  Now we confirm that the right-hand side is the correct inverse by multiplying it by $A+B$:
\begin{align*}
   (A+B)\left(A^{-1}  + A^{-1}   \sum_{i=1}^{\infty} (-BA^{-1})^{i} \right)&=
   I + \left( \sum_{i=1}^{\infty} (-BA^{-1})^{i} \right)
   +BA^{-1}  \\
   &\quad\mbox{}
   + BA^{-1}\left( \sum_{i=1}^{\infty} (-BA^{-1})^{i} \right) \\
   &=I.
\end{align*}
Finally, by submultiplicativity and the triangle inequality,
\begin{align*}
\norm{A^{-1}   \sum_{i=1}^{\infty} (-BA^{-1})^{i}}_F & \leq \norm{A^{-1}}_F \cdot \norm{BA^{-1}}_F/\left( 1 - \norm{BA^{-1}}_F\right)  \\
& \leq  \norm{B}_F \cdot\norm{A^{-1}}_F ^2/\left( 1 - \norm{B}_F \cdot \norm{A^{-1}}_F\right).
\end{align*}

% This implies:

% \begin{align*}
%     \norm{(A + B)^{-1}}_F & = \norm{A^{-1} \sum_{i=0}^{\infty} (BA^{-1})^{i}}_F \\
%     & \leq \norm{A^{-1}} (  \sum_{i=0}^{\infty} (\norm{BA^{-1}}_F)^i  )  \\
%     & = \norm{A^{-1}} (1/( 1 - \norm{BA^{-1}}_F))
% \end{align*}

% \begin{align*}
%     \norm{(A + B)^{-1}}_F & = \norm{A^{-1} \sum_{i=0}^{\infty} (BA^{-1})^{i}}_F \\
%     & \leq \norm{A^{-1}} (  \sum_{i=0}^{\infty} (\norm{BA^{-1}}_F)^i  )  \\
%     & = \norm{A^{-1}} (1/( 1 - \norm{BA^{-1}}_F))
% \end{align*}

\end{proof}

An application of the previous lemma is the following fact that is used in this section and again in the appendix.  The rationale for this lemma is that $H$ denotes a small perturbation to $K$, while $\Delta$ a small perturbation to $D$.

\begin{lem}
Suppose $K,H\in\R^{m\times n}$, $D\in\SS^n$, $D\succ 0$, $\Delta\in\SS^n$, $\Vert H\Vert_F\le \Vert K\Vert_F/2$, and $\Vert \Delta\Vert_F\cdot \Vert D^{-1}\Vert_F\le 1/2.$  Then $D+\Delta$ is invertible, and
$\Vert (K+H)(D+\Delta)^{-1}(K+H)^{\top} - KD^{-1}K^{\top}\Vert_F$ is bounded above by
\[
\frac{1}{2} \Vert K\Vert_F\cdot \Vert D^{-1}\Vert_F
\left(9\Vert K\Vert_F \cdot\Vert D^{-1}\Vert_F\cdot \Vert \Delta\Vert_F + 5\Vert H\Vert_F\right).
\]
\label{lem:KDK_lemma}
\end{lem}
\begin{proof}
First, observe by the previous lemma that $D+\Delta$ is invertible, and furthermore, if
$Z$ denotes $(D+\Delta)^{-1}-D^{-1}$, then 
$$\Vert Z\Vert_F\le 2\Vert \Delta\Vert_F\cdot \Vert D^{-1}\Vert_F^2.$$ Now observe that 
$$(K+H)(D+\Delta)^{-1}(K+H)^{\top} - KD^{-1}K^{\top}=T_1+T_2+T_3+T_4,$$
where
\begin{align*}
    T_1 &= (K+H)Z(K+H)^{\top}, \\
    T_2 &= KD^{-1}H^{\top}, \\
    T_3 &= HD^{-1}K^{\top}, \\
    T_4 &= HD^{-1}H^{\top}.
\end{align*}
Submultiplicativity on each term $T_1,\ldots,T_4$, the triangle inequality, and overestimates of the constants yield the result.  
\end{proof}

We first show that $(\Pc_3)$, $(\Pch_3)$, and $(\Pct_3)$ are NP-hard.  All three hardness results use  Peeters' construction, which is as follows.

\begin{dfn}
Starting from the input graph $\newgraph$, the {\em Peeters supergraph of $\newgraph$} is a new graph $\newgraph' := (\newvertix',\newedge')$ such that $\newvertix\subseteq \newvertix'$, $\newedge\subseteq \newedge'$, and for each pair of distinct vertices $i,j \in \newvertix$, $\newgraph'$ additionally contains four vertices $ a_{ij},b_{ij},c_{ij},d_{ij}$ and also nine edges. Their connections are shown in Fig.~\ref{fig:Hij}.  We will denote this subgraph with six vertices and nine edges by $N_{ij}$.
\end{dfn}

So, to construct the Peeters supergraph of $\newgraph$, we start from $\newgraph$, and for every pair of vertices in $\newgraph$ we add a triangular prism connecting them. It follows that $|\newvertix'| = |\newvertix| + 2 |\newvertix| (|\newvertix| - 1)$ and $|\newedge'| = |\newedge| + \dfrac{9}{2} |\newvertix| (|\newvertix| - 1)$.

% \tikzstyle{every node}=[circle, draw, fill=black!50,
%                         inner sep=0pt, minimum width=4pt]

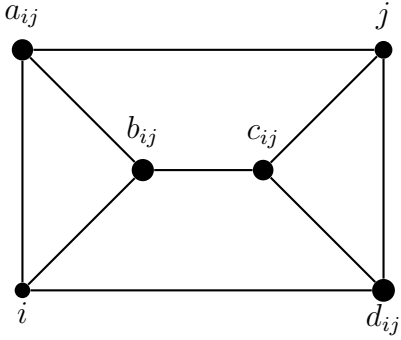
\begin{figure}[!h]
\centering
\begin{tikzpicture}[thick,scale=0.8]
\node[fill,circle,scale=0.3,label={[xshift=0.0cm, yshift=-0.7cm]$i$}] (i) at (0,0) [circle] {$i$};
\node[fill,circle,scale=0.3,label={$j$}] (j) at (6,4) [circle] {$j$};
\node[fill,circle,scale=0.3,label={$a_{ij}$}] (a) at (0,4) [circle] {$a_{ij}$};
\node[fill,circle,scale=0.3,label={$b_{ij}$}] (b) at (2,2) [circle] {$b_{ij}$};
\node[fill,circle,scale=0.3,label={$c_{ij}$}] (c) at (4,2) [circle] {$ c_{ij} $};
\node[fill,circle,scale=0.3,label={[xshift=0.0cm, yshift=-0.9cm]$d_{ij}$}] (d) at (6,0) [circle] {$ d_{ij} $};

\draw (i) edge[-] (a) (i) edge[-] (b) (a) edge[-] (b) ;
\draw (j) edge[-] (c) (j) edge[-] (d) (c) edge[-] (d) ;
\draw (a) edge[-] (j) (i) edge[-] (d)  (b) edge[-] (c);
\end{tikzpicture}
\caption{Triangular prism gadget $N_{ij}$}
\label{fig:Hij}
\end{figure}

% \tikz {
% \node (i) at (0,0) [circle] {$i$};
% \node (j) at (6,4) [circle] {$j$};
% \node (a) at (0,4) [circle] {$a_{ij}$};
% \node (b) at (2,2) [circle] {$b_{ij}$};
% \node (c) at (4,2) [circle] {$ c_{ij} $};
% \node (d) at (6,0) [circle] {$ d_{ij} $};

% \draw (i) edge[-] (a) (i) edge[-] (b) (a) edge[-] (b) ;
% \draw (j) edge[-] (c) (j) edge[-] (d) (c) edge[-] (d) ;
% \draw (a) edge[-] (j) (i) edge[-] (d)  (b) edge[-] (c);
% }

% \begin{figure}[thpb]
%     \centering
%     \includegraphics{m.png}
%     \caption{\textcolor{red}{it seems that the image fails to display because the owner is Steve}}
%     \label{fig:my_label}
% \end{figure}

Peeters argues (and it is relatively easy to convince oneself) that $\newgraph$ is 3-colorable if and only if $\newgraph'$ is 3-colorable.
Next, from $\newgraph'$, Peeters constructs a partially specified matrix $A(\newgraph') \in \SS^{\newvertix}$ such that
$$A(\newgraph')_{ij} := \begin{cases} 0, & \text{if } \{i,j\} \in \newedge(\newgraph'), \\
                             \text{unspecified but nonzero}, & \text{if } i = j, \\
                             \text{unspecified,} & \text{otherwise. }
\end{cases}$$

Peeters argues $\newgraph'$ (and therefore $\newgraph$) is 3-colorable if and only if there is a way to complete $A(\newgraph')$ over $\R$ so that $\rank(A(\newgraph'))$ is minimized and equals to three.  Note that Peeters' result applies to every field, but we only consider the field of real numbers for our purposes.

Completing $A(\newgraph')$ so that $\rank(A(\newgraph'))$ is minimized is similar to $(\Pc_3)$ with the following two differences:
\begin{enumerate}[(i)]
    \item $(\Pc_3)$ additionally requires the completion to be positive semidefinite.
    \item $(\Pc_3)$ does not provide a direct means to specify that the diagonal entries of the completion must be nonzero.
    %\item $\Pct_3$ allows small perturbations to specified ({\bf all?}) entries. 
\end{enumerate}

For (i), our formulation is already consistent with Peeters'. Since in \cite{PeetersM.J.P1996ORoF}, when $\newgraph'$ is $3$-colorable, the rank-three completion $B$ is given as
$$B_{ij} := \begin{cases} 0 \text{ (given)}  & \text{if } \{i,j\} \in \newedge', \\
                             1 & \text{if } i = j, \\
                             0 & \text{if } i,j \text{ have different colors}, \\
                             1 & \text{if } i,j \text{ have the same color.} 
\end{cases}$$

Such a $B$ is actually positive semidefinite since there exists a permutation matrix $P$ such that $P$ maps $\{1,\dots,|\newvertix'|\}$ to $ J_1 \cup J_2 \cup J_3 $ where each $J_i$ is the set of vertices with the same color and the following holds:
\begin{equation}
P^{\top} B P = \begin{bmatrix}
T_1 & \\
& T_2 \\
&& T_3
\end{bmatrix},
\label{eq:tiblocks}
\end{equation}
where each $T_i \in \SS^{J_i} = \textbf{1} \textbf{1}^{\top}$.

Therefore, to establish the NP-hardness of $(\Pc_3)$ essentially requires modification of Peeters' construction to close gap (ii). Our results further extend Peeters' because we introduce perturbations in $(\Pch_3)$ and $(\Pct_3)$. Our proof takes advantage of the requirement that the completion be semidefinite.  This is because some of Peeters' original arguments (which do not rely on semidefiniteness) do not easily handle perturbations. 
\textcolor{blue}{Peeters' original argument says that testing whether a graph is 3-colorable can be reduced to testing whether a certain matrix has a rank-3 completion. It is possible that this reduction is not preserved in the presence of perturbation.
The underlying reason for this is that Schur complement (which highly associates with matrix rank) is unstable for adverse perturbation.  Consider the following example:}
\textcolor{blue}{\begin{align*}
A := \begin{bmatrix}
    0 & 0 & 1 & 1 & 0 \\
    0 & 0 & 1 & -1 & 0 \\
    1 & 1 & 0 & 0 & 1 \\
    1 & -1 & 0 & 0 & 1 \\
    0 & 0 & 1 & 1 & 0 \\
\end{bmatrix}
\end{align*}}
\textcolor{blue}{One sees that there does not exist any $\dv$ such that $A + \Diag(\dv)$ has rank two (without semidefiniteness constraint) since the first three columns of $A + \Diag(\dv)$, which are $\begin{bmatrix}
    d_1 \\ 0 \\ 1 \\ 1 \\ 0
\end{bmatrix},\begin{bmatrix}
    0 \\ d_2 \\ 1 \\ -1 \\ 0
\end{bmatrix},\begin{bmatrix}
    {1} \\ {1} \\ d_3 \\ 0 \\ 1
\end{bmatrix}$, are linearly independent regardless of the choice for $\dv$. However, if we allow some small perturbation $\Tilde{E}$, there exist some $\Tilde{\dv},\Tilde{E}$ such that \begin{align*}
    A + \Diag(\Tilde{\dv}) + \Tilde{E} = \begin{bmatrix}
    \epsilon & 0 & 1 & 1 & \epsilon \\
    0 & \epsilon & 1 & -1 & 0 \\
    1 & 1 & \frac{2}{\epsilon} & 0 & 1 \\
    1 & -1 & 0 & \frac{2}{\epsilon} & 1 \\
    \epsilon & 0 & 1 & 1 & \epsilon \\
\end{bmatrix}
\end{align*} and $A + \Diag(\Tilde{\dv}) + \Tilde{E}$ has rank two (the last three columns are in the span of the first two columns for every nonzero $\epsilon$). 
% Note that this bad example still holds even when we have an additional semidefiniteness constraint. 
This illustrates an inherent difficulty in using Peeters' result directly. Our construction makes use of the semidefiniteness constraint in a more sophisticated way so that it is robust against small perturbations.}

% to our problems with the existence of small perturbation. When the perturbation follows certain sparsity pattern as $(\Pch_3)$, we will see later that the reduction from Peeters' result  still goes through with some modification. For general perturbation, we will develop some new mechanisms that utilize semidefiniteness constraint.

% Our method is based on a simple observation that for any positive semidefinite matrix, its $i$-th diagonal entry is zero only if the entire $i$-th row is a zero row.

\newcommand{\FirstGadget}{A}

\newcommand{\FirstOptSol}{L}
\newcommand{\FirstOptNoi}{H}

We now present the {\em $(\Pc_3)$ input construction.}  Starting from an input graph $\newgraph$, we first construct its Peeters supergraph $\newgraph'=(\newvertix',\newedge')$. 
Next, we construct $A \in \SS^{3|\newvertix'|}$ as follows:
\begin{equation}\FirstGadget  :=  I \otimes \begin{bmatrix}
1 & 1 & 1 \\
1 & 1 & 1 \\
1 & 1 & 1 
\end{bmatrix} - I,
\label{eq:3by3}
\end{equation}
where the first identity matrix is in $\SS^{\newvertix'}$ and the second identity matrix
is in $\SS^{3 |\newvertix'|}$.

The purpose of the Kronecker product is to address point (ii) above, that is, to force diagonal entries of the completion to be nonzero.
Note that every node $i$ of $\newgraph'$ is in correspondence with three rows and columns of $\FirstGadget$, say rows/columns numbered $i^{(1)},i^{(2)},i^{(3)}$. 
Define 
$$\mathcal{X}\subseteq \cE(\cK(\{1^{(1)},1^{(2)},1^{(3)},\ldots,|\newvertix(\newgraph')|^{(1)},|\newvertix(\newgraph')|^{(2)}, |\newvertix(\newgraph')|^{(3)}\}))$$ to contain the nine unordered pairs  of the form $\{i^{(\mu)},j^{(\nu)}\}$, $\mu,\nu=1,2,3$
for each edge $\{i,j\}\in \newedge(\newgraph')$.
In addition, $\mathcal{X}$ contains the three unordered pairs of the form $\{i^{(\mu)},i^{(\nu)}\}$ for each $i=1,\ldots,|\newvertix(\newgraph')|$ and for each $\mu,\nu=1,2,3$, $\mu\ne\nu$. 
 Thus, $\mathcal{X}$ contains a total of $9|\newedge(\newgraph')|+3|\newvertix(\newgraph')|$ edges.
 This concludes the specification of the input $(\FirstGadget,\mathcal{X})$ to $(\Pc_3)$.  
 %The same $(A,X)$ also serve as input to $(\Pch_3)$ and $(\Pct_3)$, and for these two problems, we additionally specify a perturbation tolerance $\epsilon_0>0$ defined as follows:
 %\begin{equation}
 %   \epsilon_0 \leq \frac{1}{2}\min \left\{  \left(\dfrac{\sqrt{2} - 1}{13 + 384(1651)^3}\right)^2, \frac{1}{n} \right\}.
 %   \label{eq:eps0def}
 %\end{equation}

The main theorem regarding this construction
is as follows.
\begin{thm}
Given a graph $\newgraph$, let $(\FirstGadget,\mathcal{X})$ be the $(\Pc_3)$ input construction for $\newgraph$ described above.  Then \begin{enumerate}
    \item If $\newgraph$ is 3-colorable, there exists
    a choice of $\FirstOptSol$ fitting $\mathcal{X}$ such that $\rank(\FirstGadget+\FirstOptSol)\le 3$
    and $\FirstGadget+\FirstOptSol\succeq 0$.
    \item 
    Conversely, there exists a universal constant
    $\epsilon_0>0$ such that the following is true.  If there is a choice of $\FirstOptSol\in \mathbb{S}^{3|\newvertix(\newgraph')|}$,
    $\FirstOptNoi\in \mathbb{S}^{3|\newvertix(\newgraph')|}$ satisfying
    the following conditions:
    \begin{enumerate}
    \item
    $\FirstOptSol$ fits $\mathcal{X}$,
    \item
    $\Vert\FirstOptNoi\Vert_F\le \epsilon_0$,
    \item
    $\FirstGadget+\FirstOptSol+\FirstOptNoi\succeq 0$, and
    \item
    $\rank(\FirstGadget+\FirstOptSol+\FirstOptNoi)\le 3$,
    \end{enumerate}
    then $\newgraph$ is 3-colorable.
\end{enumerate}
\label{thm:tilde_p3_is_nphard}
\end{thm}

\begin{rem}
This theorem simultaneously shows that $(\Pc_3)$, $(\Pch_3)$, and $(\Pct_3)$ are NP-hard.  The specification of $A$ is exact; in other words, if the graph is 3-colorable, then the first part of the theorem says that input matrix $\FirstGadget$ can be completed without any noise term.  This means that ``$\epsilon$'' in the definition of $(\Pct_3)$  and $(\Pch_3)$ can be arbitrarily small.  For example,  the ``$\epsilon$''  specified in $(\Pct_3)$ make be taken as $\epsilon_0/p(n)$; then this theorem says that the solution to $(\Pct_3)$ determines whether a graph is 3-colorable. This will be the case for all reductions in this section; the need for a perturbation to the input occurs only in the reduction given in the appendix.
\label{rem:allthreecovered}
\end{rem}

\begin{rem}
We do not explicitly compute the universal constant $\epsilon_0$, but our upcoming analysis implies that it may taken as $10^{-12}$.  The true constant is larger because the constants in our analysis are not tight.

In the proof of the theorem, $\epsilon_1$, $\epsilon_2$, \ldots denote a sequence of parameters that are all of the form $k_1\epsilon_0^{k_2}$ for positive $k_1,k_2$.  The precise values of $k_1$ and $k_2$ can be deduced from the proofs.
\end{rem}

\begin{proof}
Part (1) of the theorem is
straightforward to prove: if $\newgraph$ is 3-colorable,
then so is $\newgraph'$ as already mentioned.  For each $\{i,j\}\in \cE(\cK(\newvertix(\newgraph')))\setminus \newedge(\newgraph')$,  let $\FirstOptSol(i^{(\mu)},j^{(\nu)})=1$ for all $\mu,\nu=1,2,3$ if $i,j$ are in the same color class, else let all of these entries be 0.  Set all diagonal entries of $\FirstOptSol$ to 1.  The other entries of $\FirstOptSol$, that is, off-diagonal entries corresponding to $\newedge(\newgraph')$ or off-diagonal entries of diagonal blocks, are necessarily 0 because of the requirement that $\FirstOptSol$ must fit $\mathcal{X}$. Then $\FirstGadget+\FirstOptSol$, after permutation,
has the same form as in \eqref{eq:tiblocks}
and therefore is positive semidefinite and of rank 3.

Part (2) of the theorem requires two lemmas.
%Let $\epsilon:=2\epsilon_0$.

% let $X := (V_X,\newedge_X)$ be the graph that replaces every nodes in $\newgraph'$ with a clique of size three and denote those three vertices from $X$ in the clique corresponding to $i \in \newvertix'$ by $i^{(1)},i^{(2)},i^{(3)}$.  For all $i,j \in \newvertix'$ and for all $a,b \in \{1,2,3\}$, $ i^{(a)} j^{(b)} \in \newedge_X $ if and only if $ij \in \newedge'$.

\begin{lem}
\label{lemma_diageps}
If $\FirstOptSol$, $\FirstOptNoi$  satisfy conditions 2(a)--(d) of Theorem~\ref{thm:tilde_p3_is_nphard},
 then $|(\FirstGadget+\FirstOptSol+\FirstOptNoi)_{ii} - 1| \leq \epsilon_2$ for all $i \in \{1,2,\ldots,3|\newvertix(\newgraph')|\}$.
\end{lem}

\begin{proof}
The diagonal entries are of $\FirstGadget+\FirstOptSol+\FirstOptNoi$ are in correspondence with $i^{(\mu)}$ as $i$ ranges over $\newvertix(\newgraph')$ and $\mu\in\{1,2,3\}$.  Fix a particular $i$, and let us analyze the entry corresponding to $i^{(2)}$ (the other cases are symmetric).
There always exists a clique in $\newgraph'$ of size three that contains $i$ thanks to the Peeters supergraph construction. Denote vertices in this clique by $i,j,k$. Consider the $4$-by-$4$ submatrix of $\FirstGadget+\FirstOptSol+\FirstOptNoi$ indexed by rows $i^{(1)},i^{(2)},j^{(1)},k^{(1)}$ and columns $i^{(2)},i^{(3)},j^{(2)},k^{(3)}$. Denote this submatrix by $B_1 + H_1$ where 
$$B_1 := \begin{bmatrix}
    1 & 1 & 0 & 0 \\
    x & 1 & 0 & 0 \\
    0 & 0 & 1 & 0 \\
    0 & 0 & 0 & 1 
\end{bmatrix},H_1 := \begin{bmatrix}
    H_1^{(1)} & H_1^{(2)} \\
    H_1^{(3)} & H_1^{(4)}
\end{bmatrix}, H_1^{(i)} \in \R^{2 \times 2}, \norm{H_1^{(i)}}_F \leq \epsilon_0.$$
The scalar $x$ in the above formulation is a diagonal entry and therefore corresponds to $\FirstOptSol(i^{(2)},i^{(2)})+\FirstOptNoi(i^{(2)},i^{(2)})$ since the diagonal entries of $\FirstGadget$ are all 0's.  Without loss of generality, we may assume that $x=\FirstOptSol(i^{(2)},i^{(2)})$ and $\FirstOptNoi(i^{(2)},i^{(2)})=0$. We can further assume that $x>0$ since $\FirstGadget+\FirstOptSol+\FirstOptNoi\succeq 0$ and row $i^{(2)}$ is not all zeros.

 Compute the Schur complement of the lower right $2$-by-$2$ block of $B_1+H_1$:
 $$S := \begin{bmatrix}
     1 & 1 \\
     x & 1
 \end{bmatrix} + \underbrace{H_1^{(1)} - H_1^{(2)} \left(I + H_1^{(4)}\right)^{-1} H_1^{(3)}}_{:= T}.$$
  By triangle inequality and the fact that $\norm{\cdot}_F$ is submultiplicative, it follows
that
 \begin{equation}
 \norm{T}_F \leq \norm{H_1^{(1)  }}_F + \norm{H_1^{(2)  }}_F \norm{H_1^{(3)}}_F \norm{ \left(I + H_1^{(4) }\right)^{-1}}_F \leq \epsilon_0 + \epsilon_0^2 \norm{ \left(I + H_1^{(4) }\right)^{-1}}_F. 
 \end{equation}
 
Since $\norm{H_1^{(4)}}_F < 1$,  express $(I + H_1^{(4) } )^{-1} = I +  \sum\limits_{i=1}^{\infty} \left(- H_1^{(4) }\right)^{i}$ and by triangle inequality again,
 \begin{equation}
 \norm{T}_F \leq \epsilon_0 + \epsilon_0^2 \left(2 + \sum\limits_{i=1}^\infty \epsilon_0^i\right) \leq \epsilon_0 + \epsilon_0^2 \left(2 + \frac{\epsilon_0}{1 - \epsilon_0}\right) \leq \epsilon_1,
 \end{equation}
where $\epsilon_1=2\epsilon_0$, assuming $\epsilon_0$ is sufficiently small (as we have assumed).

If $\rank(A + \FirstOptSol + \FirstOptNoi) \leq 3$, then by Lemma~\ref{lemma_schur} it is necessary that $\rank(S) = 1$ and therefore the first row of $S$ is linearly dependent on the second row of $S$ and it follows that 
\begin{align*}
    \dfrac{1 + T_{11}}{ x } = \dfrac{1 + T_{12}}{ 1 + T_{22}}  & \implies x = \dfrac{ (1+T_{11}) (1+ T_{22}) }{1 + T_{12}} \\
    & \implies \dfrac{(1-\epsilon_1)^2}{(1+\epsilon_1)} \leq x \leq \dfrac{(1+\epsilon_1)^2}{(1-\epsilon_1)} \\
    & \implies 1 - \epsilon_2 \leq x \leq   1 + \epsilon_2,
\end{align*}
where $\epsilon_2=4\epsilon_1$, assuming $\epsilon_1$ is sufficiently small.
\end{proof}

\begin{lem}
\label{lemma_one_zero_minusone}
Assume $\FirstOptSol,\FirstOptNoi$ satisfy conditions 2(a)--(d) of Theorem~\ref{thm:tilde_p3_is_nphard}.
Suppose $\{i,j\}\in \cE(\cK(\newvertix(\newgraph)))$ and $\mu,\nu\in\{1,2,3\}$.   Let $t :=(\FirstGadget+\FirstOptSol+\FirstOptNoi)_{i^{(\mu)},j^{(\nu)}}$. 
Then

(a) One of the following inequalities must hold:
\begin{itemize}
    \item 
$|t- 1| \leq 1/20$,
\item $|t+1| \le 1/20$, or
\item 
 $|t| \leq 1/20$.
\end{itemize}

(b) Assume further that
$\{i,j\}\in \newedge(\newgraph)$. Then the third inequality of part (a) holds.
\end{lem}

\begin{rem}
Note that the pair of vertices $i,j$ in this lemma %and the next 
are assumed to be
in the original vertex set $\newvertix$, i.e., they are not internal vertices of the Peeters gadget $N_{ij}$.
\end{rem}

\begin{proof}
Before turning to part (a), we quickly dispense with part (b).  Part (b) follows since $\FirstOptSol(i^{(\mu)},j^{(\nu)})=0$ if $\{i,j\}\in \newedge(\newgraph)\subseteq \newedge(\newgraph')$ and $\FirstOptSol$ fits $\mathcal{X}$. In addition, $\FirstGadget(i^{(\mu)},j^{(\nu)})=0$ in this case, also by construction.  Therefore, $|t|=|\FirstOptNoi(i^{(\mu)},j^{(\nu)})|$, and this number is at most $\epsilon_0$ hence also at most $1/20$.

Now, turning to part (a), consider the $6\times 6$ principal submatrix in $\FirstGadget + \FirstOptSol + \FirstOptNoi$ indexed by $i^{(\mu)}$, $a_{ij}^{(\mu)}$, $b_{ij}^{(\mu)}$, $d_{ij}^{(\mu)}$ $j^{(\nu)}$, $c_{ij}^{(\mu)}$. By Lemma \ref{lemma_diageps}, this submatrix has the form of $B_3 + H_3$ where
$$ B_3 := \begin{bmatrix}
       1 & 0 & 0 & 0 & u & v \\
       0 & 1 & 0 & w & 0 & x \\
       0 & 0 & 1 & y & z & 0 \\
       0 & w & y & 1 & 0 & 0 \\
       u & 0 & z & 0 & 1 & 0 \\
       v & x & 0 & 0 & 0 & 1 \\
\end{bmatrix},$$
$\norm{H_3}_F \leq 6 \epsilon_2$, $B_3 + H_3 \succeq 0$, $u,v,w,x,y,z \in \R.$
In this submatrix, $t=u+H_3(5,1)$. 
We will show that $u$ is close to one of $-1$, $0$, or $1$; in this case the result will follow since $|t-u|\le 6\epsilon_2$.

Before investigating this $6\times 6$ matrix, consider one of its $4\times 4$ submatrices
$$B_4 := B_3(1:4,1:4) =  \begin{bmatrix}
        1 & 0 & 0 & 0 \\
        0 & 1 & 0 & w \\
        0 &0 & 1 & y \\
        0 & w & y & 1 \\
\end{bmatrix}.$$ 
Note that $B_4$'s eigenvalues are $1$, $1$, $1 - (w^2 + y^2)^{1/2}$, $1+(w^2 + y^2)^{1/2}$. Since $(B_3 + H_3)({1:4},{1:4})$ is positive semidefinite and singular, it is necessary that $|1 - (w^2 + y^2)^{1/2}| \leq \norm{H_3}_F\leq 6\epsilon_2$  (see \cite[Cor.~8.1.6]{GolubVanLoan}). Assuming $\epsilon_2$ is sufficiently small, this implies, first, that $w^2+y^2 \leq 2,$ hence $w^2 \leq 2, y^2 \leq 2$. Furthermore, $ 1 - (w^2 + y^2)^{1/2} = \dfrac{1 - (w^2 + y^2)}{1 + (w^2 + y^2)^{1/2}}$, then $|1 - (w^2 + y^2)| \leq 6 \epsilon_2 (1+ (w^2 + y^2)^{1/2}) \leq 15 \epsilon_2$.
A similar analysis can be applied to $(w,y)$, $(u,z)$, $(v,x)$, $(u,v)$, $(w,x)$ and $(y,z)$. 

Next, consider the $4\times 4$ principal submatrix of $B_3 + H_3$ indexed by $1,2,5,6$. 
This submatrix can be written in the form $B_5+H_5$ where
$$B_5:=\left(\begin{array}{cccc}
1 & 0 & u & v \\
0 & 1 & 0 & x \\
u & 0 & 1 & 0 \\
v & x & 0 & 1
\end{array}
\right),$$
and $\Vert H_5\Vert_F\le 6\epsilon_2$.  Let the eigenvalues of $B_5$ be denoted $\lambda_1,\ldots,\lambda_4$.  Since $B_5+H_5$ is singular by the rank-3 assumption, it follows that one of the eigenvalues, say $\lambda_4$, must satisfy $|\lambda_4|\le \Vert H_5\Vert_F$. On the other hand, all four of them satisfy $|\lambda_i|\le \Vert B_5\Vert_F$.  Since $u^2 + v^2 \le 2$ and $v^2 + x^2 \le 2$, it follows that $\Vert B_5\Vert_F\le 3$.  Thus, we have the following chain of inequalities:
\begin{align*}
    |\det(B_5)| & = |\lambda_1\lambda_2\lambda_3\lambda_4| \\
    & \le 3^3\cdot(6\epsilon_2)= 162\epsilon_2.
\end{align*}
We can write $|\det(B_5)|$ in closed form:
$$|\det(B_5)|=|(1-u^2)(1-v^2-x^2)-u^2v^2|.$$
We have shown that $|1-u^2|\le 1$, while $|1-v^2-x^2|\le 15\epsilon_2$.
Thus,
\begin{align*}
    162\epsilon_2 &\ge |\det(B_5)| \\
    &=|(1-u^2)(1-v^2-x^2)-u^2v^2| \\
    &\ge u^2v^2 - 15\epsilon_2.
\end{align*}
Thus, we have arrived at the inequality $u^2v^2\le \epsilon_3$, where $\epsilon_3 := 177\epsilon_2$.
This means that either $u^2\le \sqrt{\epsilon_3}$ or $v^2\le \sqrt{\epsilon_3}$.   In other words, either $|u|\le \epsilon_3^{1/4}$ or $|v|\le \epsilon_3^{1/4}$.

In the first case, we have established an upper bound on $|u|$, thus also establishing an upper bound of $\epsilon_3^{1/4}+6\epsilon_2$ on $|t|$.  This shows that the third case of part (a) of the lemma holds, assuming that $\epsilon_0$ is chosen sufficiently small so that $\epsilon_3^{1/4}+6\epsilon_2\le 1/20.$  In the case that $|v|\le \epsilon_3^{1/4}$, we know that $|1-u^2-v^2|\le 15\epsilon_2$, which implies that $|1-u^2|\le 15\epsilon_2 + |v|^2\le 15\epsilon_2+\epsilon_3^{1/2}.$
In other words, $|1-u|\cdot|1+u|\le 15\epsilon_2+\epsilon_3^{1/2}$, which means that either $|1-u|$ or $|1+u|$ is bounded above by $(15\epsilon_2+\epsilon_3^{1/2})^{1/2}$. This implies that the first or second case of part (a) holds, assuming that $(15\epsilon_2+\epsilon_3^{1/2})^{1/2}\le 1/20$.  This is assured by a sufficiently small choice of $\epsilon_0$. 
\end{proof}

{\em Conclusion of proof of Theorem~\ref{thm:tilde_p3_is_nphard}.}
We have already showed that if the graph $\newgraph$ is 3-colorable, then there exists a choice of $\FirstOptSol$ fitting $\mathcal{X}$ such that $\rank(\FirstGadget+\FirstOptSol)\le 3$, $\FirstGadget+\FirstOptSol\succeq 0$.  So to complete the proof, we need to show the opposite direction, namely, if there exists a completion satisfying conditions 2(a)--(d) of the theorem, then $\newgraph$ is 3-colorable.  Therefore, assume $\FirstOptSol,\FirstOptNoi$ exist that satisfy these conditions.  For two nodes $i,j\in \newvertix(\newgraph)$, write $i\sim j$ if $i=j$ or if
the first or second condition of part (a) of Lemma~\ref{lemma_one_zero_minusone} apply to 
$t=(\FirstGadget + \FirstOptSol+\FirstOptNoi)_{i^{(1)},j^{(1)}}$
(i.e., $t\approx \pm 1$ rather than $t\approx 0$).
Clearly, `$\sim$' defined in this manner is reflexive and symmetric.  We claim also that it is transitive.  Suppose, e.g., that $i\sim j$ and $j\sim k$; we wish to prove that $i\sim k$.  Assume $i,j,k$ are distinct, else the result is obvious.  By the lemma, the $(i^{(1)},j^{(1)},k^{(1)})$ principal submatrix of $\FirstGadget + \FirstOptSol+\FirstOptNoi$ may be written as  $B_6+H_6$, where $$B_6=\left(\begin{array}{ccc}
1 & x & z \\
x & 1 & y \\
z & y & 1
\end{array}
\right),$$
and $\Vert H_6\Vert_F \le 3/20$. Here, $x=\pm 1$, $y=\pm 1$, and $z\in\{-1,0,1\}$ since $x+H_6(2,1)$ is the quantity denoted as $t$ in the lemma associated with $\{i^{(1)},j^{(1)}\}$, $y$ is associated with $\{j^{(1)},k^{(1)}\}$, and $z$ is associated with $\{i^{(1)},k^{(1)}\}$.  However, we claim that if $z=0$, then $B_6+H_6$ cannot be semidefinite, contradicting the assumption.  To see this, suppose $z=0$, and observe that the three eigenvalues of $B_6$ are $1-\sqrt{2},1+\sqrt{2},1$ (irrespective of the choice of signs for $x$ and $y$).  On the other hand,
$\Vert H_6\Vert_F\le 3/20 <\sqrt{2}-1$.  Therefore, $B_6+H_6\not\succeq 0$ in the case $z=0$.  Thus, we conclude that $z=\pm 1$, implying that $i\sim k$.

Thus, `$\sim$' is an equivalence relationship, and therefore it partitions the nodes of $\newvertix(\newgraph)$ into $m$ equivalence classes $\newvertix_1,\ldots,\newvertix_m$ for some $m\ge 1$.  Observe that no edge can join two distinct nodes $i,j$ in the same class $\newvertix_r$ since $i\sim j$
implies $(\FirstGadget + \FirstOptSol+\FirstOptNoi)_{i^{(1)},j^{(1)}}\approx \pm 1$, but this is precluded by part (b) of the lemma if $\{i,j\}\in \newedge(\newgraph)$.  Thus, $\newvertix_1,\ldots,\newvertix_m$ define a valid $m$-coloring of $\newgraph$.  
The last claim to prove, therefore, is that $m\le 3$.  Suppose not; let $i_1,\ldots,i_4$ lie in $\newvertix_1,\ldots,\newvertix_4$ respectively.  Consider the principal submatrix of
$\FirstGadget + \FirstOptSol+\FirstOptNoi$ indexed by $(i_1^{(1)},\ldots,i_4^{(1)})$, which has the form $I+H_7$, where $I$ denotes the $4\times 4$ identity matrix
and $\Vert H_7\Vert_F \le 4/20$. The first term can be written as $I$ because all the off-diagonal entries of the $(i_1^{(1)},\ldots,i_4^{(4)})$-principal submatrix of $\FirstGadget + \FirstOptSol+\FirstOptNoi$ are close to zero by the hypotheses that $i_1\not \sim i_2$, $i_1\not\sim i_3$, etc., meaning that the third case of part (a) of the lemma applies to the off-diagonal entries.
Since $\Vert H_7\Vert_F<1$, it follows that $I+H_7$ is nonsingular and hence is of rank 4, contradicting the assumption.  This concludes the proof of theorem.
\end{proof}

We next turn to the NP-hardness of $(\Pc_1)$, $(\Pch_1)$, and $(\Pct_1)$.
In this paragraph we present the {\em $(\Pc_1)$ input construction.}
Let $\newgraph$ be an undirected graph, and let $(\FirstGadget_0,\mathcal{X})$ be the $(\Pc_3)$ input construction
described in the paragraphs preceding Theorem \ref{thm:tilde_p3_is_nphard}.  Recall that 
$\FirstGadget_0\in\SS^n$, where we define $n:=3|\newvertix(\newgraph')|$ and where $\newgraph'$ is the Peeters supergraph of $\newgraph$.  Recall also that each vertex $i\in \newvertix(\newgraph')$ corresponds to three rows/columns, say $i^{(1)},i^{(2)}, i^{(3)}$ of $\FirstGadget_0$.  Let $i$ be called the {\em parent} of these three rows/columns.  We denote the parent mapping by $\pi(\cdot)$, i.e., $\pi(i^{(\mu)})=i$ for $\mu=1,2,3$.
Recall that $(i'',j'')\in \mathcal{X}$ means either that $(\pi(i''),\pi(j''))$ is an edge in $\newgraph'$ or that $\pi(i'')=\pi(j'')$ but $i''\ne j''$.

We call a pair $\{i',j'\}\in \cE(\cK(\newvertix(\newgraph'))\setminus \newedge(\newgraph'))$ a {\em nonedge} of $\newgraph'$.
Let $A\in\SS^n$ be defined as
$$A(i'',j'')=\left\{\begin{array}{ll}
2 & \mbox{if $\{\pi(i''),\pi(j'')\}$ is a nonedge of $\newgraph'$}, \\
0 & \mbox{if $\{\pi(i''), \pi(j'')\}\in \newedge(\newgraph')$}, \\
1 & \mbox{if $\pi(i'')=\pi(j'')$ and $i''\ne j''$}, \\
0 & \mbox{if $i''=j''$.}
\end{array}\right.$$
Thus, $A$ has nearly the same definition as $A_0$ except that its entries corresponding to nonedges are 2's instead of 0's.

The nonedges of $\newgraph'$ each correspond to nine symmetric pairs of 2's in $A$.  These nine pairs of 2's implicitly define an undirected graph whose nodes are the rows/columns of $\FirstGadget$ and whose edges correspond to its unspecified entries.  Call this graph $\newgraph''$, so that $|\newvertix(\newgraph'')|=n=3|\newvertix(\newgraph')|$, and let $m:=|\newedge(\newgraph'')|=9
\left(\choose{|\newvertix(\newgraph')|}{2}-|\newedge(\newgraph')|\right)$.
Let $K \in \{0,1\}^{n \times m}$ be the node-edge incidence matrix of $\newgraph''$.   Finally,
define $B\in\SS^{m+n}$ by 
\begin{equation}
    B := \begin{bmatrix}
    0 & K^{\top} \\
    K &  A
\end{bmatrix} + k  I 
\label{eq:bdef}
\end{equation}
where 
\begin{equation}
    k:= 2 \max\limits_{v \in \newvertix(\newgraph'')} \{ \text{degree}(v) \} + 1.
    \label{eq:kdef}
\end{equation}

\begin{thm}
\label{thm_P1_tilda}
Let $B$ be constructed from input graph $\newgraph$ using the $(\Pc_1)$ input construction in the preceding paragraphs.  Then
\begin{itemize}
    \item If $\newgraph$ is 3-colorable, there exists $\dv\in\R_+^{n+m}$ such that $B-\Diag(\dv)\succeq 0$ and $\rank(B-\Diag(\dv))\le m+3$.
    \item
    Conversely,
    if there exists $\dv \in \R_+^{n + m}$ and $H \in \SS^{n+m}$ such that $\norm{H}_F \leq \epsilon_5$, $B-\Diag(\dv)+H\succeq 0$, and $\rank(B-\Diag(\dv)+H)\le m+3$, then $\newgraph$ is 3-colorable.
\end{itemize}
\end{thm}
Here, $\epsilon_5=c_1\epsilon_0k^{-\sigma}m^{-\tau}$, where $\epsilon_0$ is the universal tolerance in Theorem~\ref{thm:tilde_p3_is_nphard}.  Also, $c_1$, $\sigma$ and $\tau$ are universal constants whose values can be deduced from the forthcoming proof.

\begin{rem}
Note that this proof shows that all three of $(\Pc_1)$, $(\Pch_1)$, and $(\Pct_1)$ are NP-hard for the same reason noted in Remark~\ref{rem:allthreecovered}.
\end{rem}

\begin{proof}
First, suppose $\newgraph$ is 3-colorable, and fix a particular 3-coloring.  Recall that this 3-coloring extends to $\newgraph'$.
Let the color of $i'\in \newvertix(\newgraph')$ be denoted $c(i')\in\{1,2,3\}$.
Then we choose $\dv:=[\dv^\newedge;\dv^\newvertix]\in\R^{m+n}$ as follows.
For an edge $e=\{i'',j''\}\in \newedge(\newgraph'')$,  take
$$d^\newedge_e=\left\{\begin{array}{ll}
k-1, & \mbox{if $c(\pi(i''))=c(\pi(j''))$}, \\
k-0.5, & \mbox{if $c(\pi(i''))\ne c(\pi(j''))$}.
\end{array}\right.$$
For each $i''\in \newvertix(\newgraph'')$, take $d^\newvertix_{i''}=k-n_1(i'')-2n_2(i'')$, where 
\begin{align*}
n_1(i'')&=|\{e\in \newedge(\newgraph''): \mbox{$e=\{i'',j''\}$ is incident on $i''$ and $c(\pi(i''))=c(\pi(j''))$}\}|, \\
n_2(i'')&=|\{e\in \newedge(\newgraph''): \mbox{$e=\{i'',j''\}$ is incident on $i''$ and $c(\pi(i''))\ne c(\pi(j''))$}\}|. 
\end{align*}
The choice of $k$ in \eqref{eq:kdef} ensures that $\dv^\newedge\ge \mathbf{0}$ and $\dv^\newvertix\ge \mathbf{0}$.

We claim that the Schur complement of the $m\times m$ leading principal block of $B-\Diag(\dv)$ has form \eqref{eq:tiblocks}, and therefore $\rank(B-\Diag(\dv))=m+3$.  
Observe that the leading $m\times m$ block of $B-\Diag(\dv)$ is $kI-\Diag(\dv^\newedge)$, while the remaining $n\times n$ principal diagonal block is $A+kI-\Diag(\dv^\newvertix)$.
Therefore,  the Schur complement of the $(1,1)$ block equals $A+kI-\Diag(\dv^\newvertix)-K(kI-\Diag(\dv^\newedge))^{-1}K^\top$.  We will show that this matrix has the form of $\eqref{eq:tiblocks}$.

The $(1,1)$ block of $B-\Diag(\dv)$, $kI-\Diag(\dv^\newedge)$, is a diagonal matrix each of whose diagonal entries is either 1 or $1/2$ by construction of $B$ and $\dv^\newedge$.  By construction, the `1' entries correspond to pairs $\{i'',j''\}$ of rows/columns of $A$ such that $\pi(i'')$ and $\pi(j'')$ have  the same color, while the `1/2' entries correspond to pairs with different colors.  
Therefore, $K(kI-\Diag(\dv^\newedge))^{-1}K^\top$ is an $n\times n$ matrix that is the sum of $2\times 2$ principal submatrices, one submatrix for each $\{i'',j''\}\in \newedge(\newgraph'')$.  This submatrix is composed of four 1's at positions $(i'',i''),$ $(i'',j'')$, $(j'',i'')$, and $(j'',j'')$ if $(\pi(i''),\pi(j''))$ have the same color.  The submatrix has four 2's at these positions if  $(\pi(i''),\pi(j''))$  have different colors.

Since the entries of $A$ in the off-diagonal positions of these submatrices are all 2's, 
subtraction of $K(kI-\Diag(\dv^\newedge))^{-1}K^\top $  leaves a Schur complement  
$A+kI-\Diag(\dv^\newvertix)-K(kI-\Diag(\dv^\newedge))^{-1}K^\top$
whose off-diagonal entries are 1's if the parents of $i'',j''$ are in the same color class or 0's if they are in different color classes.  Thus, the off-diagonal positions of this
Schur complement indeed 
correspond to \eqref{eq:tiblocks}.

As for the diagonal entries of the Schur complement, one checks that we have constructed entries of $\dv^\newvertix$ equal exactly 1 plus the corresponding diagonal entry of $K(kI-\Diag(\dv^\newedge))^{-1}K^\top$, and thus the Schur complement has all 1's on the diagonal.  This concludes the argument that the Schur complement has the same structure as \eqref{eq:tiblocks} and thus has rank 3.  The same argument also assures semidefiniteness of the Schur complement, and therefore of the entire matrix $B-\Diag(\dv)$.

Conversely,  suppose there exists $\dv \in \R_+^{n + m}, H \in \SS^{n+m}, \norm{H}_F \leq \epsilon_5$ such that $B -\Diag(\dv) + H \succeq 0 $ and $\rank(B -\Diag(\dv) + H) \leq 3 + m$. Partition $H := \begin{bmatrix}
    H_1 & H_2^\top \\
    H_2 & H_3
\end{bmatrix}$ so that $H_1 \in \SS^m, H_2 \in \R^{n \times m}, H_3 \in \SS^{n}$.

The first step of the argument is to obtain a lower bound (\eqref{eq:b11lb} below) on the pivots in the leading principal $m\times m$ block.  Select an $e=\{u,v\}\in \newedge(\newgraph'')$.  Consider the $2\times 2$ principal submatrix of $B-\Diag(\dv)+H$ indexed by $\{e,u\}$  (i.e., $e$ is an index into the first $m$ rows/columns of $B-\Diag(\dv)+H$, while $u$ is an index into the last $n$ rows/columns).  This submatrix has the form 
$$B_1 := \begin{bmatrix}
    k- d_{\{u,v\}}+h_1 & 1+h_2 \\
    1+h_2 & k - d_{u} +h_3 
\end{bmatrix},$$
with $h_1^2+2h_2^2+h_3^2\le \epsilon_5^2\le 1/16$.
Here, $B_1(1,1)$ is a diagonal entry of $kI-\Diag(\dv^\newedge)+H_1$, $B_1(1,2)=B_1(2,1)$ comes from $K+H_2$, and $B_1(2,2)$ is a diagonal entry of $A+kI-\Diag(\dv^\newvertix)+H_3$.
Since $B_1\succeq 0$,
\begin{align}
 \det(B_1) \geq 0
&\implies (k - d_{\{u,v\}}+h_1) (k - d_{u}+h_3) \geq (1+h_2)^2\notag\\
& \implies k - d_{\{u,v\}}+h_1 \geq \frac{(1+h_2)^2}{k-d_u+h_3} \notag\\
& \implies k-d_{\{u,v\}}+h_1 \ge \frac{1}{2(k+1)},\label{eq:b11lb}
\end{align}
where the last line follows since $k-d_u+h_3\le k+1$ and $1+h_2\ge 3/4$.

Write $H_1=H_1^D+H_1^O$, where $H_1^D$ contains the diagonal entries and $H_1^O$ contains the off-diagonal entries.
The Schur complement of the $(1,1)$ block of $B+H-\Diag(\dv)$ is
\begin{align*}
S&:=A+kI-\Diag(\dv^\newvertix)+H_3 - (K+H_2)(kI-\Diag(\dv^\newedge)+H_1^D +H_1^O)^{-1}(K^{\top}+H_2^{\top}) \\
&= \bar{S} + H_3+T
\end{align*}
where
$$\bar{S}=A+kI-\Diag(\dv^\newvertix)-K(kI-\Diag(\dv^\newedge)+H_1^D)^{-1}K^{\top},$$
and
$\Vert T\Vert_F$ is at most
$$ 
\frac{1}{2} \Vert K\Vert_F \cdot \Vert (kI-\Diag(\dv^\newedge)+H_1^D)^{-1}\Vert_F
\cdot
\left(9\Vert K\Vert_F \cdot\Vert (kI-\Diag(\dv^\newedge)+H_1^D)^{-1}\Vert_F
\cdot \Vert H_1^O \Vert_F + 5\Vert H_2\Vert_F\right)
$$
by Lemma~\ref{lem:KDK_lemma}.  The lemma is applicable because by assumption $\Vert H_1^O\Vert_F\le\epsilon_5$ and\\
$\Vert (kI-\Diag(\dv^\newedge)+H_1^D)^{-1}\Vert_F\le 2\sqrt{m}(k+1)$ 
by \eqref{eq:b11lb}, assuming we impose the
restriction that $\epsilon_5\le 1/(4\sqrt{m}(k+1)).$

Furthermore, $\Vert H_2\Vert_F\le \epsilon_5$ while $\Vert K\Vert_F= \sqrt{2m}$.  Substituting these bounds into the preceding inequality yields:
$$\Vert T\Vert_F \le \sqrt{2}m(k+1) \left[18\sqrt{2}m(k+1)+5\right] \epsilon_5\le 40 m^2 (k+1)^2 \epsilon_5.
$$

Note that $\bar{S}$ is of the form $A_0+L$, where $A_0$ is matrix used in the proof of Theorem~\ref{thm:tilde_p3_is_nphard}, and
$L=kI-\Diag(\dv^\newvertix)-K(kI-\Diag(\dv^\newedge)+H_1^D)^{-1}K^{\top}$.  The nonzero entries of $L$ are exactly on the diagonal and in the positions of $A$ not indexed by $\mathcal{X}$, thus indicating that $L$ fits $\mathcal{X}$.

%Each of the remaining terms $T_1,\ldots,T_5$ is small because it involves a product of factors each of which is bounded and one of which tends to 0 like $\epsilon_5$.  In more detail,
%$\Vert T_1\Vert_F\le \epsilon_5$ because $H_3$ is a submatrix of $H$.  As for $T_2$, we observe that
%$$\Vert T_2\Vert_F \le \Vert K\Vert_F \cdot\Vert(kI-\Diag(\dv^E)+H_1)^{-1}\Vert_F \cdot \Vert H_2\Vert_F.$$
%The first factor equals $\sqrt{2m}$ since $K$ contains $m$ columns each of which contains two 1's.  The second factor is bounded by $2(k+2)m^{1/2}$ as in 
%\eqref{eq:invkIH1}.  The third factor is bounded by $\epsilon_5$ by assumption.  Therefore, $\Vert T_2\Vert\le \mbox{const}\cdot\epsilon_5\cdot km$.  A similar analysis applies to $T_3$.  Term $T_4$ has a smaller bound since it contains two factors of $H_2$, but we can overestimate it with the same bound used for $T_2$.  Finally, $\Vert T_5\Vert_F\le \Vert K\Vert_F^2\cdot \Vert Z\Vert_F\le 2m\epsilon_6$, where here we have applied 
%\eqref{eq:zub2}.

Thus, we conclude that the Schur complement $S$ of the $(1,1)$ block of $B-\Diag(\dv)+H$ is written as
$A_0+L+H'$ where $H'=H_3+T$ and therefore $\Vert H'\Vert_F\le \epsilon_5 + 40 m^2 (k+1)^2 \epsilon_5$.  By choosing a suitable $\epsilon_5$, we are assured that this bound is less than $\epsilon_0$.  We have already explained that $L$ fits $X$.  As the Schur complement of a positive definite leading principal submatrix, it follows that $A_0+L+H'\succeq 0$.

%it follows that
%\begin{align*}
%& (K + H_2) (k I - \Diag(\dv_E) + H_1)^{-1} (K + H_2)^{\top} 
%= (K + H_2) ((k I - \Diag(\dv_E))^{-1} + Z) (K+H_2)^{\top} \\
%& = K (k I - \Diag(\dv_E))^{-1} K^{\top} + \\
%& \underbrace{K (k I - \Diag(\dv_E))^{-1} H_2^{\top} + H_2^{\top} %(k I - \Diag(\dv_E))^{-1} K^{\top} + H_2 (k I - \Diag(\dv_E))^{-1} %H_2^{\top} +  (K + H_2) Z (K + H_2)^{\top}}_{:=T}
%\end{align*}

%Here $Z$ comes from Lemma \ref{lem_bound_of_inverse_of_sum} and 
%\begin{align*}
%\norm{Z}_F & \leq \norm{H_2}_F (\norm{(kI - %\Diag((\dv_E))^{-1}}_F)^2 / ( 1 - \norm{H_2}_F (\norm{(kI - %\dv_E)^{-1}}_F)) \\
%& \leq ( m k^2 \norm{H}_F) / (1 - \norm{H}_F \sqrt{m} k)
%\end{align*}
%and there
%\begin{align*}
%\norm{T}_F & \leq 2 \sqrt{2} k m \norm{H}_F + \sqrt{m} k %(\norm{H}_F)^2 + ( \sqrt{m} k + \norm{H}_F)^2 ( m k^2 \norm{H}_F) / %(1 - \norm{H}_F \sqrt{m} k) \\
%& \leq 2 \sqrt{2} k m \norm{H}_F + 0.5 (\norm{H}_F) + 8 m^2k^4 %\norm{H}_F
%\end{align*}
%where the last line uses the fact that $\norm{H}_F \leq %\sqrt{m}k,\norm{H}_F \leq 1/(2\sqrt{m}k)$.

Now applying Lemma \ref{lemma_schur}, it follows that
\begin{align*}
3+m &\ge \rank(B -\Diag(\dv) + H)  \\
&= \rank( k I - \Diag(\dv^\newedge) + H_1 ) + \rank(S)  \\
    & = m + \rank(A_0+L+H').
\end{align*}
Thus, $\rank(A_0+L+H')\le 3$.  Since $L$ fits $\mathcal{X}$ and $\Vert H'\Vert_F\le \epsilon_0$,
Theorem~\ref{thm:tilde_p3_is_nphard} states that $\newgraph$ is 3-colorable.
\end{proof}

Finally, we prove that $(\Pc_2)$ and $(\Pch_2)$ are NP-hard.  The  {\em $(\Pc_2)$ input construction} that we now present is quite similar to the $(\Pc_1)$ input construction that precedes Theorem~\ref{thm_P1_tilda}.  In particular, we use the same choice of $A$ and $K$ as before.  Finally, we take
 $B\in\SS^{m+n}$ to be 
$$B := \begin{bmatrix}
    0 & K^{\top} \\
    K &  A
\end{bmatrix},$$
which is the same as the $(\Pc_1)$ input construction except for the omission of the additive term $kI$.

\begin{thm}
\label{thm_P2}
Let $B$ be constructed from input graph $\newgraph$ using the $(\Pc_2)$ input construction from the previous paragraph.  Then
\begin{itemize}
    \item If $\newgraph$ is 3-colorable, there exists $\dv\in\R_+^{n+m}$ such that $B+\Diag(\dv)\succeq 0$ and $\rank(B+\Diag(\dv))\le m+3$.
    \item
    Conversely,
    if there exists $\dv \in \R_+^{n + m}$ and $H \in \SS^{n+m}$ such that $\norm{H}_F \leq \epsilon_0$,
        $B +\Diag(\dv) + H \succeq 0 $, $\rank(B +\Diag(\dv) + H) \leq 3 + m$, and $\mathrm{nz}(H)\subseteq\mathrm{nz}(B)$
        then $\newgraph$ is 3-colorable.
\end{itemize}
\end{thm}

Here $\epsilon_0$ is the universal constant in Theorem~\ref{thm:tilde_p3_is_nphard}.

\begin{rem}
Note that this proof shows that both $(\Pc_2)$ and $(\Pch_2)$ are NP-hard.  The theorem does not apply to $(\Pct_2)$ because of the restriction in the second claim on the nonzero entries of $H$.  We return to this point after the proof is complete.  
\end{rem}

\begin{proof}
The proof of the first claim is nearly identical the the proof of the first claim of Theorem~\ref{thm_P1_tilda}.  Assuming $\newgraph$ is 3-colorable, we fix a 3-coloring of $\newgraph'$ and we select diagonal entries for the $(1,1)$ block of either $1/2$ or $1$ depending whether the endpoints of the corresponding nonedge are of different or the same color.  The resulting matrix has rank at most $m+3$ for the same reason as in the proof of Theorem~\ref{thm_P1_tilda}.

For the proof of the second claim, we partition $H$ and $\dv$ in the same manner as in the proof of Theorem~\ref{thm_P1_tilda}. By the assumption that $\mathrm{nz}(H)\subseteq \mathrm{nz}(B)$, $H_1=0$.  
Similar to Theorem~\ref{thm_P1_tilda}, we can write the Schur complement of the $(1,1)$ block of $B+\Diag(\dv)+H$ as
\begin{align*}
S&:=A+\Diag(\dv^\newvertix)+H_3 - (K+H_2)\Diag(\dv^\newedge)^{-1}(K+H_2)^\top,
\end{align*}
with $\Vert H_3\Vert_F\le \epsilon_0$.
Now we observe that $L:=\Diag(\dv^\newvertix)-(K+H_2)\Diag(\dv^\newedge)^{-1}(K+H_2)^\top$ can have nonzero entries only in the entries of $S$ corresponding to nonedges of $\newgraph'$ because $K$ and $H_2$ have the same sparsity pattern.  Therefore, the additive term can affect only the entries of $A+H_3$ that were already allowed to be arbitrary in the second claim of Theorem~\ref{thm:tilde_p3_is_nphard}, i.e., $L$ fits $\mathcal{X}$.  Thus, by the second claim Theorem~\ref{thm:tilde_p3_is_nphard},
if $\rank(S)\le 3$ and $S\succeq 0$, then $\newgraph$ is 3-colorable. 
\end{proof}

We now explain why the technique used in Theorem~\ref{thm_P1_tilda} to show that $(\Pct_1)$ is NP-hard does not immediately extend in the preceding theorem to show that $(\Pct_2)$ is NP-hard.  In the proof that $(\Pct_1)$ is hard, we needed a bound on the perturbation terms in $S$, the Schur complement of the $(1,1)$ block.  For this purpose, we used Lemma~\ref{lem:KDK_lemma}. For this lemma, we needed an upper bound on the entries of $(kI-\Diag(\dv^\newedge)+H_1^D)^{-1}$, i.e., a lower bound on the diagonal entries of $kI-\Diag(\dv^\newedge)+H_1^D$.  This lower bound came from 
\eqref{eq:b11lb}, which relied in turn on an upper bound of $k+1$ on $B_1(2,2)=k-d_u+h_3$.  The reason there is an upper bound in on $B_1(2,2)$ in $(\Pct_1)$ is that a nonnegative vector $\dv$ is subtracted from the diagonal entries of the given matrix.  In contrast, in $(\Pct_2)$, a nonnegative entry is added to the diagonal entries, and hence the corresponding matrix $B_1(2,2)$ in the context of $(\Pct_2)$ has no prior upper bound.  This means that there is no prior lower bound on the pivots, meaning there is no prior upper bound on the Schur-complement update term arising in the proof of Theorem~\ref{thm_P1_tilda}.

Thus, we need to construct a different $B$ for $(\Pct_2)$ and a more complicated argument to show that even in the presence of small pivots in the $(1,1)$ block, if $\newgraph$ is not 3-colorable, then there is no rank-3 completion of the Schur complement. This proof is deferred to the appendix.

\end{section}

\begin{section}{Completeness in the first-order theory of the reals}
\label{sec:existsR}
%\textbf{write the whole section so that it is a single repeat of our}

In Section~\ref{sec:alg}, we saw that we can solve problems ($\Pc_1$) and ($\Pc_2$) by solving a series of systems of polynomial equations. Those algorithms may need to solve exponentially many polynomial systems in the worst case. However, since the decision version of ($\Pc_2$) that asks whether the optimal value is $r$, can be formulated as
\[
U U^{\top} - \Diag(\dv) = A,
\]
where the variables are $U \in \R^{n \times r}$ and $\dv \in \R^n$, we see that solving these problems is no harder than solving some specially structured polynomial systems with the number of variables and equations bounded by a polynomial function of $n$ (and the degree of each polynomial is at most 2). In this section, we will show that $(\Pc_2)$ and $(\Pc_3)$ are complete for the complexity class of solving real polynomial system, also called first-order theory of the reals and usually denoted $\exists\R$ in the literature.  This yields another proof that these two problems are NP-hard and further establishes that both problems lie in $\textup{PSPACE}$ since $\textup{NP}\subseteq \exists\R\subseteq \textup{PSPACE}$, the latter inclusion a result of Canny \cite{Canny1988}.  The results in this section are independent of those in the previous section because the results in this section assume exact input data and are therefore not amenable to perturbation.  On the other hand, the results in this section exactly characterize the complexity of $(\Pc_2)$ and $(\Pc_3)$, whereas the previous section established NP-hardness but not an exact characterization of the complexity.
In this section, the problems under consideration are decision problems.  In other words, $(\Pc_1)$ is understood as: given $A$ and $r$, does there exist a $\dv\ge\mathbf{0}$ such that $\rank(A-\Diag(\dv))\le r$ and $A-\Diag(\dv)\succeq 0$?  An algorithm to solve this problem is required merely to correctly output yes/no without necessarily producing the vector $\dv$.  The same applies for $(\Pc_2)$ and $(\Pc_3)$.

We start with a claim that $(\Pc_3)$ is reducible to $(\Pc_2)$, and therefore a proof of the completeness of $(\Pc_3)$ immediately extends to $(\Pc_2)$.  (The opposite reduction is trivial, since $(\Pc_2)$ is a special case of $(\Pc_3)$.)  Our argument does not, however, extend to $(\Pc_1)$ for reasons explained below.  Note that the results in the previous section do not imply a reduction of $(\Pc_3)$ to $(\Pc_2)$ because the theorems in that section started with a graph-coloring instance rather than a generic instance of $(\Pc_3)$.

The reduction of $(\Pc_3)$ to $(\Pc_2)$ again involves an off-diagonal block called $K$  constructed in the previous section, which is a node-edge incidence matrix all of whose entries were either 0 or 1.  Taking the Schur complement of the $(1,1)$ block causes $K$ to subtract arbitrary nonnegative numbers from the specified off-diagonal entries of $A$.  For a fully general reduction of $(\Pc_3)$ to $(\Pc_2)$, we need to either add or subtract numbers from the specified off-diagonal entries.  This is accomplished by taking a Schur complement twice, once with $K$ and a second time with $\bar K$, which is a node-arc (rather than node-edge) incidence matrix for the rows/columns of the unspecified entries of $A$.  

We now formally present this construction and the theorem establishing its correctness.
Given an instance of $(\Pc_3)$, that is,
$A \in \SS^n$, $\mathcal{X}\subseteq\cE(\cK(\{1,\ldots,n\}))$, let $\mathcal{X}^c$ denote\\
$\cE(\cK(\{1,\ldots,n\}))\setminus \mathcal{X}$, $m:=|\mathcal{X}^c|$, and
$$B \in \SS^{2m + n} := \begin{bmatrix}
    0 & 0 & K^{\top} \\ 0 & 0 & \xoverline{K}^{\top} \\ K & \xoverline{K} & A
\end{bmatrix}$$
be an instance of $(\Pc_2)$.  
Here, $K \in \{0,1\}^{n \times m}$ is the node-edge incidence matrix of $\mathcal{X}^c$ (in other words, one row per index in $\{1,\ldots,n\}$, one column per entry $\{i,j\}\in \mathcal{X}^c$, and exactly two 1's in this column in positions $i$ and $j$).  On the other hand, $\xoverline{K} \in \{-1,0,1\}^{n \times m}$ is the node-arc incidence matrix of $\mathcal{X}^c$ after assigning an arbitrary orientation to the edges.
The following theorem states that this construction
reduces $(\Pc_3)$ to $(\Pc_2)$.

\begin{thm}
\label{thm:P3toP2}
Let $(A,\mathcal{X},r)$ be input data for $(\Pc_3)$ ($r\in\{1,\ldots,n\})$, let
$B$ be the input matrix constructed in the preceding paragraph, and consider input data $(B,2m+r)$
for $(\Pc_2)$.  Then,
\begin{itemize}
    \item[]
(i) there exists an $R\in\SS^n$ such that $R$ fits $\mathcal{X}$,
$A+R\succeq 0$, and $\rank(A+R)\le r$ 
\end{itemize}

if and only if

\begin{itemize}
    \item[](ii) there exists a $[\uv;\vv;\dv]\in\R^m\times\R^m\times\R^n$ such that
$B+\Diag([\uv;\vv;\dv])\succeq 0$ and
$\rank(B+\Diag([\uv;\vv;\dv]))\le 2m+r$.
\end{itemize}
\label{thm:P3redP2}
\end{thm}

\begin{proof}
% \textcolor{red}{mix with undirected graph and directed graph}
\underline{(ii)$\Rightarrow$(i).} Suppose
$[\uv;\vv;\dv] \in \R^{2m+n} $ satisfies
$B+\Diag([\uv;\vv;\dv])\succeq 0$ and 
$\rank(B+\Diag([\uv;\vv;\dv]))\le 2m+r$ for some $r$.
Let $R:= \Diag(\dv) - K \Diag(\uv)^{-1} K^{\top} - \bar{K}\Diag(\vv)^{-1}\bar K^\top$.  Since the entries of $\uv,\vv$ are associated with $\{i,j\}\in \mathcal{X}^c$, assign subscripts to these entries as $u_{ij},v_{ij}$ respectively for $\{i,j\}\in \mathcal{X}^c$.

Under hypothesis (ii), necessarily $\uv>\mathbf{0}$ and
$\vv>\mathbf{0}$ since every column of $K$ has at least one nonzero entry.  By 
Lemma \ref{lemma_schur},
$$\begin{bmatrix}
    \Diag(\uv) & 0 & K^{\top} \\ 0 & \Diag(\vv) & \xoverline{K}^{\top} \\ K & \xoverline{K} & A + \Diag(\dv)
\end{bmatrix} \succeq 0$$
implies $A + R \succeq 0$.  Furthermore, 
and $2m+r \ge \rank(B + \Diag([\uv;\vv;\dv])) = \rank(\Diag(\uv)) + \rank(\Diag(\vv)) + \rank(  A + R) = 2 m + \rank(  A + R)$.
Notice that $R$ fits $\mathcal{X}$ and, in particular, has the following  structure:
$$R_{ij} = \begin{cases} \displaystyle d_i - \sum\limits_{il \in \delta(i)} \left(\dfrac{1}{u_{il}} + \dfrac{1}{v_{il}}\right),  & \text{if } i = j, \\
 - \left(\dfrac{1}{u_{ij}} - \dfrac{1}{v_{ij}}  \right),  & \text{if } \{i,j\} \in  \mathcal{X}^c, \\
 0, & \text{otherwise.}
\end{cases}$$

%Such observation illustrates that given an optimal solution of $(\Pc_2)$ with input $B$, we can construct a feasible solution of $\Pc_3$ with input $A$ and its objective value is $r - 2m$. We claim that $R$ is an optimal solution of $\Pc_3$.

%Suppose there exists another $R' \in \SS^n$ that $R'$ fits $\newgraph$ and $A + R' \succeq 0$ and $\rank(A + R') < r - 2m$.

\underline{(i)$\Rightarrow$(ii).} Suppose that 
$R\in \SS^n$, $R$ fits $\mathcal{X}$, $A + R \succeq 0$ and $\rank(A + R) \le r$.
We construct $[\uv;\vv;\dv] \in \R^{2m+n}$ as follows:
\begin{equation}
\begin{aligned}
     & v_{ij} := \dfrac{1}{|R_{ij}| + 1} , \forall \{i,j\} \in  \mathcal{X}^c, \\
     & u_{ij} :=  \dfrac{v_{ij}}{1 - v_{ij} R_{ij}} , \forall \{i,j\} \in \mathcal{X}^c, \\
     & d_{i} :=  R_{ii} + \sum\limits_{ij \in \delta(i)} \left(\dfrac{1}{u_{ij}  } + \dfrac{1}{v_{ij} } \right)    , \forall i \in \{1,\dots,n\}.
\end{aligned}
\label{eq:vud}
\end{equation}
Note that the denominator of $u_{ij}$ is assuredly positive because of the definition of $v_{ij}$.
Under this construction, we have $\Diag(\dv) - K \Diag(\uv)^{-1} K^{\top} - \bar{K}\Diag(\vv)^{-1}\bar{K}^\top = R$. Applying Lemma \ref{lemma_schur}, it follows that $B + \Diag([\uv;\vv;\dv] ) \succeq 0$ and $\rank(B + \Diag([\uv;\vv;\dv] ))=2m+\rank(A+R)\le 2m+r.$
\end{proof}

Next, we explain why $(\Pc_3)$ lies in $\exists\R$.  Suppose $(A,\mathcal{X},r)$ is input data for $(\Pc_3)$.  Observe that $(\Pc_3)$ may be written as: do there exist $L\in\SS^n$ and $U\in\R^{n\times r}$ such that $L$ fits $\mathcal{X}$ and $A+L=UU^{\top}$?  Both conditions, namely ``$L$ fits $\mathcal{X}$" and ``$A+L=UU^{\top}$," are polynomial equations in the unknowns $L,U$.

The main result of this section is the other direction, namely, the reduction of existence of a solution to an arbitrary polynomial system to $(\Pc_3)$.
Our proof is based on Shitov's result in \cite{shitov2016hard}. Our proof works for both Turing-machine model (which is the usual definition of $\exists\R$) and a real-number model of computation (see, e.g., \cite{Blum1989OnAT}) as it only involves elementary arithmetic operations.  In the Turing machine case, one assumes that the input to $(\Pc_3)$ (i.e., specified matrix entries) consists of integer data, while in the real-number case, the input to $(\Pc_3)$ may contain arbitrary real numbers.
Shitov shows that finding a solution of any real polynomial system can be reduced to finding rank-3 matrix completion of certain symmetric matrix. A related idea appeared in \cite{mnev1988universality} that for every positive integer $d$, every semialgebraic set in $\R^d$ is ``\emph{stably equivalent}\footnote{for a definition, see \cite{mnev1988universality}}'' to the set of all geometric (matrix) realizations of a rank-3 oriented matroid.

The reason that Shitov's result does not apply directly is that his reduction specifies some diagonal entries of the input matrix, whereas for $(\Pc_3)$, all diagonal entries must be unspecified.  Therefore, our presentation follows Shitov's construction while focusing particular attention on how our construction diverges from Shitov's.  Assume we are given a system of polynomial equalities and inequalities with vector of variables $\xv\in\R^n$. A polynomial inequality of the form $p(\xv)\ge 0$ can be replaced by $p(\xv)=z^2$, where $z$ is a new variable.  Similarly, an inequality of the form $p(\xv)>0$ may be replaced by $z^2p(\xv)=1$, where again $z$ is a new variable.  Thus, we may assume that our system is written $f_1(\xv)=\cdots=f_t(\xv)=0$ where $F=(f_1,\ldots,f_t)$ is a sequence of polynomials.

For any monomial $p = \xi x_{i_1} \dots x_{i_k} $ (with $\xi \in \Z$ in the Turing machine model or $\xi\in\R$ in the real-number model) appearing in one of the $f_s$'s, we define
$$\sigma(p) := \{ \pm 1, \pm \xi, \pm x_{i_1}, \pm x_{i_1} x_{i_2}, \dots, \pm x_{i_1} \dots x_{i_k}, p \},$$
that is, a sequence of $2k+5$ monomials.

For a general real polynomial $f = p_1 + \dots + p_s$ written as a sum of monomials, we define the following sequence of polynomials:
$$\sigma(f) := \sigma(p_1) \cup \sigma(p_2) \cup \dots \cup \sigma(p_s) \cup \{ 0, \pm p_1, \pm (p_1 + p_2), \dots ,\pm f \} \cup \{ \pm x_1 , \dots, \pm x_n \}  $$
 and for the input sequence of real polynomials $F = \{ f_1,f_2,\ldots,f_t \}$, we define $\sigma(F) := \sigma(f_1) \cup \sigma(f_2) \cup  \dots \cup  \sigma(f_t) $.

Let $\sigma = \sigma(F)$, and let $\sigma^3$ denote all triples of elements in $\sigma$. We denote by $\mathcal{H} = \mathcal{H}(F)$ the set of those 3-vectors in $\sigma^3$ that have one or more entries equal to $1$ or $-1$. 

For this proof, we need four types of matrices.  A plain symbol like $A$ denotes a matrix in $\R^{m\times n}$ or $\SS^n$.   The notation
$A[\xv]$ denotes a matrix whose entries are polynomials in unknowns $x_1,\ldots,x_n$ and whose coefficients are in $\Z$ in the Turing-machine case or in $\R$ in the real-number model case.  Given such a polynomial matrix, if $\xiv\in\R^n$ is a vector of numbers, then $A[\xiv]$ denotes the numerical matrix resulting from substituting $\xiv$ for $\xv$.  

The notation $A^*$ denotes a matrix whose entries either are in the base ring ($\Z$ or $\R$) or are `$*$', meaning `unspecified.'  Finally, if $A^*$ is a matrix with some unspecified entries, then $A^\#$ is a matrix in $\R^{m\times n}$ or $\SS^n$ (i.e., with real-number entries) that represents a completion of $A^*$, in other words, each `$*$' in $A^*$ is replaced by a real number.

Denote by $U[\xv]$ the $|\mathcal{H}|\times 3$ matrix whose columns are the vectors in $\mathcal{H}$ in some order and define $W[\xv] =  U[\xv]^{\top}U[\xv]$.
Shitov constructs a symmetric $|\mathcal{H}|\times|\mathcal{H}|$ matrix $B^*$ in following way:
$$ B^*(\uv,\vv) := \begin{cases} 0 & \text{ if } W[\xv](\uv,\vv) \in F \\
W[\xv](\uv,\vv) & \text{ if } W[\xv](\uv,\vv) \text{ is a constant (i.e., a degree-0 polynomial)} \\
* & \text{otherwise.} 
 \end{cases} $$
 In this definition, $\uv,\vv$ represent two elements of $\mathcal{H}$.
 
 \begin{rem}
The rank of any completion of $B^*$ will be at least 3 because the columns of the $3\times 3$ identity matrix appears as elements of $\mathcal{H}$.
 \end{rem}
 
 Shitov's main result establishes an equivalence between completing $B^*$ and solving the polynomial system.  These results are as follows.
 
 \begin{lem}
 \label{shitov thm}
 (see Lemma 11 in \cite{shitov2016hard}) Let $P,L$ be $3 \times |\mathcal{H}|$ matrices over $\R$ such that the matrix $P^{\top} L$ is a completion of $B^*$. Let $C$ be the matrix obtained by taking the columns of $L$ with indices
in $E = \{(1, 0, 0), (0, 1, 0), (0, 0, 1)\}\subseteq \mathcal{H}$.  Then we have $C^{\top} P = C^{-1} L = U[\xiv]$ 
where $\xiv$ is a solution of the system of polynomial equations $f_1 = 0, \ldots, f_t = 0$. \end{lem}
 
 \begin{cor} (see Corollary 12 in \cite{shitov2016hard})
 \label{cor_shitov}
System of polynomial equations $f_1 = 0, \dots, f_t = 0$ has a solution if and only if $B^*$ admits a rank-3 completion with respect to $\R$ .
 \end{cor}
 
 One direction of Shitov's proof is straightforward: If $\xiv$ is a solution to $f_1=0, \ldots, f_t=0$, then $B^\#:=W[\xiv]=U[\xiv]^{\top}U[\xiv]$ is a valid completion of $B^*$.  The other direction of his proof shows inductively that a completion of $B^*$ must have the form $W[\xiv]$ for some $\xiv$.  The proof
 proceeds from constants to monomials to polynomials, showing that each step of the completion is essentially forced.  Since $B^*(\uv,\vv)=0$ for $(\uv,\vv)$ such that $W[\xv](\uv,\vv)$ is a polynomial in the original system, this induction shows that a valid completion exists if and only if the original polynomials can simultaneously be set to 0 by $\xiv$.
 
 Corollary \ref{cor_shitov} does not immediately apply to $(\Pc_3)$ because the constructed matrix $B^*$ may have specified diagonal entries (i.e., diagonal entries not equal to `$*$') and therefore is not a valid input instance of $(\Pc_3)$.

Let $\xoverline{\mathcal{H}}$ denote the extension of $\mathcal{H}$ to the multiset
$\xoverline{\mathcal{H}}:=\mathcal{H}\cup\mathcal{E}$ where
$\mathcal{E}=\{\ev_1,\ev_1,\ev_2,\ev_2,\ev_3,\ev_3 \}$ (i.e., two copies of each column of the $3\times 3$ identity matrix) so that
$|\xoverline{\mathcal{H}}|=|\mathcal{H}|+6$.
Construct $3\times |\xoverline{\mathcal{H}}|$ matrix $\xoverline{U}[\xv]$  and $\xoverline{W}[\xv] = \xoverline{U}[\xv]^{\top} \xoverline{U}[\xv]$ from $\xoverline{\mathcal{H}}$ in the same way that $U[\xv]$, $W[\xv]$ were constructed from $\mathcal{H}$. Furthermore,  construct $\xoverline{B}^*$ with respect to $\xoverline{W}[\xv]$ in the same way that $B^*$ was constructed with respect to $W[\xv]$. However, unlike the construction of $B^*$, we make all diagonal entries of $\xoverline{B}^*$ unspecified even if they are constant.  In this way, $\xoverline{B}^*$ is a valid input instance for $(\Pc_3)$.

%  \begin{rem}
%  \label{new lower bound}
% The optimal value of $\Pc_3$ with input $\xoverline{\mathcal{B}}$ is $3$ when $F$ has solution.
%  \end{rem}

\begin{thm}
The system of polynomial equations $f_1 = 0, \dots, f_t = 0$ has a solution over $\R$ iff $\xoverline{B}^*$ can be completed to a rank-$3$ positive semidefinite matrix (i.e., a solution to $(\Pc_3)$ is obtained).
\label{thm:red_poly_p3}
\end{thm}

\begin{proof}
The forward direction is straightforward (and is the same as the forward direction in \cite{shitov2016hard}).  Suppose $f_1(\xiv)=\cdots=f_t(\xiv)=0$ for some $\xiv\in \R^n$.  Define $\xoverline{B}^\#:=(\xoverline{U}[\xiv])^{\top}\xoverline{U}[\xiv]$.  Then
$\xoverline{B}^\#$  defined in this manner is indeed a valid completion of $\xoverline{B}^*$.  In particular, $\xoverline{B}^*$ and $\xoverline{B}^\#$ agree in all specified entries corresponding to inner products of degree-0 3-vectors.  Since $f_i(\xiv)=0$ for all $i=1,\ldots,t$, the entries of $\xoverline{B}^\#$ corresponding to these polynomials are equal to 0 and hence equal to the specified entries of $\xoverline{B}^*$.

The remainder of the proof focuses on the reverse direction, and therefore, assume there exists a completion of $\xoverline{B}^*$ that is a rank-3 positive semidefinite matrix.  Let $\ev_{i}^{(j)}$ for $i=1,\ldots,3$, $j=1,\ldots,3$ denote the three copies of $\ev_{i}$  in $\xoverline{\mathcal{H}}$ (two in the extension $\mathcal{E}$ and one original in $\mathcal{H}$), and let $\mathcal{E}'$ be this set of nine entries of $\xoverline{\mathcal{H}}$.

%Without losing generality, we may assume that there is no $u \in \mathcal{H}$ such that $B^*(u,u) = 0$.
 %When there indeed exists such $u$ where $B^*(u,u) = 0$, since the final completion of $B^*$ is positive semidefinite, this implies that the off-diagonal entries of column $u$ in $B^*$ are either $0$ or $*$. This cannot happen because each columns of $\xoverline{U}$ has a $1$ or $-1$ hence makes a nonzero constant inner product with $e_1,e_2$ or $e_3$. 

Let $\xoverline{R}\in\R^{3 \times |\xoverline{\mathcal{H}|}}$ be such that $\xoverline{R}^{\top} \xoverline{R}$ is a completion of $\xoverline{B}^*$.
Note $\xoverline{B}^*({\mathcal{E}',\mathcal{E}'})$ has the form
$$\begin{bmatrix}
    B_1^* & 0 & 0  \\
    0 & B_2^* & 0  \\
    0 & 0 & B_3^*
\end{bmatrix} \text{ where } B_i^* := \begin{bmatrix}
    * & 1 & 1 \\ 
    1 & * & 1 \\
    1 & 1 & *
\end{bmatrix},\forall i \in \{1,2,3\}.$$
Since $ \xoverline{R}^{\top} \xoverline{R}$ is a rank-3 completion, this means each $B_i^*$ is completed to be a rank-1 matrix. 

\begin{lem}
There is a unique way to complete $B_i^*$ to a rank-$1$ semidefinite matrix, namely, setting each $*$ to be $1$.
\label{lem:uniquecompletion3by3}
\end{lem}

\begin{proof}
Let $B_i^\#$ be the rank-1 semidefinite completion of $B_i^*$. 
Suppose for a contradiction that $B_i^\#(1,1)<1$.  Then it must hold that $B_i^\#(2,2)>1$ because $\det(B_i^\#(1:2,1:2))\ge 0$.  Then a contradiction is obtained by considering whether $B_i^\#(3,3)\le 1$ or $B_i^\#(3,3)\ge 1$ as follows. If $B_i^\#(3,3)\le 1$ then one sees that $\det(B_i^\#([1,3],[1,3]))<0$.  If $B_i^\#(3,3)\ge 1$, then one sees that $\rank(B_i^\#(2:3,2:3))=2$.   Either case contradicts the assumption that $B_i^\#$ is positive semidefinite and rank-1.  A similar contradiction arises if we start by assuming $B_i^\#(1,1)>1$.  Finally, by symmetry among the three rows, the same contradiction arises if $B_i^\#(2,2)\ne 1$ or $B_i^\#(3,3)\ne 1$.
\end{proof}

Let $J\subseteq \xoverline{\mathcal{H}}$ denote the set $\{\ev_1^{(1)},\ev_2^{(1)},\ev_3^{(1)}\}$, where the superscript $^{(1)}$ without loss of generality indexes a copy of $\ev_i$ in the extension $\mathcal{E}$ rather than the original $\mathcal{H}$. By the analysis above, we have $(\xoverline{R}^{\top} \xoverline{R})(J,J) = I$. In this case, there exists an orthogonal 
matrix $\xoverline{C} \in \R^{3 \times 3}$ such that $(\xoverline{C} \xoverline{R} )(:,J) = I$ and $(\xoverline{C} \xoverline{R})^{\top} \xoverline{C} \xoverline{R} = \xoverline{R}^{\top} \xoverline{R}$. Thus we can assume
without loss of generality that $\xoverline{C}$ is the identity matrix. 
In this case, we have $\xoverline{R}({:,\ev_1^{(1)}}) = [1,0,0]^{\top},\xoverline{R}({:,\ev_2^{(1)}}) = [0,1,0]^{\top},\xoverline{R}({:,\ev_3^{(1)}}) = [0,0,1]^{\top}$.

Now, let $\uv \in {\mathcal{H}}$ such that $B^*(\uv,\uv)$ is a constant rather than `$*$'.  (Note that the previous sentence refers to the original $(\mathcal{H},B^*)$  rather than $(\xoverline{\mathcal{H}},\xoverline{B}^*)$.)  
%First, observe that this constant is nonzero.  This follows because every $\uv\in\mathcal{H}$ whose all three entries are constants has at least one of them equal to $\pm 1$.
%Therefore, $B^*(\uv,\uv)>0$.
We claim that $(\xoverline{R}^{\top} \xoverline{R})(\uv,\uv) = B^*(\uv,\uv)$.
Since each entry of $\uv$ is a constant and $B^*(\uv,\ev_i^{(1)})$ is an off-diagonal entry
of $B^*$, (since $\uv \in \mathcal{H}$ and $\ev_i^{(1)} \in \mathcal{E}$),
we have $B^*(\uv , \ev_1^{(1)} ) =  \xoverline{B}^*(\uv , \ev_1^{(1)}  ) = u_1$, $B^*(\uv, \ev_2^{(1)} ) =  \xoverline{B}^*(\uv , \ev_2^{(1)}  ) = u_2$, $B^*(\uv, \ev_3^{(1)} ) =  \xoverline{B}^*(\uv , \ev_3^{(1)}  ) = u_3$.
Since we also have $\xoverline{R}({:,\ev_1^{(1)}}) = [1,0,0]^{\top}$, $\xoverline{R}({:,{\ev_2^{(1)}}}) = [0,1,0]^{\top},\xoverline{R}({:,\ev_3^{(1)}}) = [0,0,1]^{\top}$ and the equality
$\xoverline{R}(:,\uv)^{\top}\xoverline{R}(:,\ev_i^{(1)})=\xoverline{B}^*(\uv,\ev_i^{(1)})$,
this establishes that $\xoverline{R}(:,\uv) = (u_1,u_2,u_3)$ and therefore $(\xoverline{R}^{\top} \xoverline{R})(\uv,\uv) = B^*(\uv,\uv)$.

In this case, we know that $(\xoverline{R}^{\top} \xoverline{R})({\mathcal{H},\mathcal{H}})$ is a valid rank-three completion of $B^*$.  This is because it was already a valid completion for $\xoverline{B}^*$, which encompasses all the specified (i.e., non-$*$) off-diagonal entries of $B^*$, and the argument in the last paragraph showed that it is also valid for the specified diagonal entries of $B^*$. Therefore by Lemma \ref{shitov thm}, this completion defines a
a solution of the system of polynomial equations $f_1 = 0, \ldots, f_t = 0$.
\end{proof}

\begin{rem}
\label{rm_double}
There exists an integral matrix $A \in \SS^n$ with $\diag(A) = \mathbf{0}$ and a positive integer $r$ such that every solution $\dv^*$ to $(\Pc_2)$ posed with data $(A,r)$ is doubly exponential in the size of the input. 
\end{rem}

\begin{proof}
Consider the polynomial system $x_1 = 2$ and $x_t = x_{t-1}^2$ for all $t \in \{2,3,\dots,n\}$. For this polynomial system, there is exactly one real solution and in this solution, $x_n = 2^{2^{n-1}}$. In Shitov's construction, some diagonal entries of the final completion are $1 + x_i^2,\forall i \in \{1,2,\dots,k\}$. This means the optimal solution of $(\Pc_2)$ with input $A$ can be doubly exponential in the size of input.
\end{proof}

The reduction presented in this section does not extend to showing the completeness of $(\Pc_1)$ for $\exists\R$.  The reason is that in order to accommodate doubly-exponential diagonal entries that could arise in the construction, the initial given $A$ would need to have equally large entries on the diagonal (since the formulation of $(\Pc_1)$ allows only subtraction of positive numbers from the diagonal).  These large entries require an exponential amount of space, and therefore the reduction overall would not be polynomial time.

%\textbf{mv's result; instance matrix without Shitov}

\end{section}

\begin{section}{Conclusion}
Our algorithms for problems $(\Pc_1)$ and $(\Pc_2)$ run in polynomial time provided their optimal values (minimum rank of the positive semidefinite part of the decomposition) are bounded above by an absolute constant. However, since these algorithms require the solution of certain systems of multivariate polynomial equations (whose degrees, number of variables and number of constraints grow with that optimal value), for many instances of $(\Pc_1)$ and $(\Pc_2)$ these algorithms cannot be expected to be efficient in practice. One possible practical approach is to remove the part of the algorithms which involves solving nonlinear systems of multivariate polynomial equations and replace the enumeration of all subsets of size $r$ with a suitable data adaptive heuristic. Indeed, Algorithm~\ref{alg1} just requires solving a linear system of equations and a Cholesky decomposition. Algorithm~\ref{alg1} can be implemented for improved performance by exploiting special structures (e.g., sparsity, a priori information on ranks of certain submatrices) in the given problem instances.   

Our hardness results leave at least %\sout{three} 
\textcolor{blue}{four} natural questions unanswered:
\begin{enumerate}
    \item Is $(\Pc_1)$ complete for $\exists\R$?
    \item Our completeness results establish equivalence for the feasibility question. Can these results be strengthened to show some kind of equivalence in terms of the solution sets (analogous to Mn\"{e}v's related result)? I.e., does every semialgebraic set arise as the solution set of $(\Pc_2)$?
    $(\Pc_3)$?
    \item Do problems $(\Pc_1)$ and $(\Pc_2)$ remain NP-hard when their data are restricted to $0,1$ matrices?
    \item \textcolor{blue}{Do the hardness results extend to number systems other than $\R$?  Note that Peeters' hardness proof applies to any field, and Shitov's applies to any commutative ring.  Since our results pertain to positive definiteness, we would also require, minimally, an ordering on the base field.}
\end{enumerate}

\end{section}

\section{Acknowledgement}
\textcolor{blue}{The authors are grateful to two anonymous referees for their careful reading of the paper and helpful comments that have significantly improved the presentation.}

% \begin{section}{First-order heuristic}
% \label{sec:first-order}

% \end{section}

\appendix
\input{perturbedp2.tex}

\addcontentsline{toc}{chapter}{References}
\bibliographystyle{plain}
\bibliography{lowrankdecomp}
%\nocite{*}

\end{document}

%% file: perturbedp2.tex
\renewcommand\d{{\bm{d}}}
\newcommand\bz{{\bm{0}}}
\newcommand\eps{\epsilon}
%\newcommand\rank{\mathop{\mathrm{rank}}}
%\newcommand\Diag{\mathop{\mathrm{Diag}}}
%\begin{document}
%\begin{center}
% \Large NP-hardness proof of Perturbed P2
%\end{center}

\section{NP-hardness of \texorpdfstring{$(\Pct_2)$}{Approx (P2)}}

\subsection{Restatement of \texorpdfstring{$(\Pct_2)$}{Approx (P2)}}

A polynomial $p(n)$ is fixed in advance.  One is given $B\in\SS^n$ such that $\diag(B)=\bz$, an integer $r$, and a number $\epsilon>0$.  Assume that there
exists a vector $\dv_0\ge \bz$ and matrix $H_0$ such that
$\Vert H_0\Vert_F\le \eps$,
$B+H_0+\Diag(\dv_0)\succeq 0$, and $\rank(B+H_0+\Diag(\dv_0))\le r$.  Find
a  vector $\dv\ge \bz$ such that there exists $H$ such that
$\Vert H\Vert_F\le p(n)\eps$,
$B+H+\Diag(\dv)\succeq 0$, and $\rank(B+H+\Diag(\dv))\le r$.

The construction in this appendix takes as input an undirected graph $\mathcal{G}$ and produces a matrix $B$, integer $r$, and number $\eps>0$.  If the graph is 3-colorable, then there exists $\dv$ and $H$ such that $B+H+\Diag(\dv)$ is semidefinite and has rank $r$, and $\Vert H\Vert_F\le \eps$.  Therefore, a candidate algorithm that correctly solves $(\Pct_2)$ must find some $H$ and $\dv$ satisfying
$B+H+\Diag(\dv)\succeq 0$, $\rank(B+H+\Diag(\dv))\le r$, and $\Vert H\Vert_F\le p(n)\eps$.  On the other hand, if $\mathcal{G}$ is not 3-colorable, our construction has the property that there is no ($H,\dv)$ satisfying the properties in the previous sentence, so the candidate algorithm that correctly solve $(\Pct_2)$ must report failure on such an instance.  In this way we prove that any candidate algorithm for $(\Pct_2)$ can solve the NP-hard problem of determining graph 3-colorability.

\subsection{Preliminary graph construction}
\label{subsec:robustcoloring}

Suppose we are given a graph $\mathcal{G}$ that we wish to test for 3-colorability.
Let $c>0$ be a fixed integer.
From $\mathcal{G}$ it is possible to construct a larger graph $\mathcal{G}'$ with the following
property.  If $\mathcal{G}$ is 3-colorable, then so is $\mathcal{G}'$.  However, if $\mathcal{G}$ is
not 3-colorable, then neither is any induced subgraph of $\mathcal{G}'$ whose number
of nodes is
$|\mathcal{V}(\mathcal{G}')|-c$ or larger.  In other words, $\mathcal{G}'$ is ``robustly'' non-3-colorable in the
sense that non-3-colorability is preserved  under deletion of up to $c$ nodes.

The construction of $\mathcal{G}'$ is as follows. Replace every node in $\mathcal{G}$ with
a gadget consisting of
$3(c+1)$ nodes. Connect these nodes with a complete 3-partite graph.
In other words, partition the nodes of the gadget into three sets of size
$(c+1)$ each, and then connect every node to all the nodes in the other two
partitions.  This requires a total of $3(c+1)^2$ edges.  The point is that in
any 3-coloring of the gadget, the three partitions must be three different colors,
and this property holds robustly under the deletion of any $c$ nodes.

Now, for each original node of $\mathcal{G}$,
pick one of the three partitions of its gadget to be ``exposed''.  For every edge $\{i,j\}$
of the original graph $\mathcal{G}$, join all $c+1$ exposed copies of $i$ to all
$c+1$ exposed copies of $j$, so that the original edge of $\mathcal{G}$ is replaced by
$(c+1)^2$ edges in $\mathcal{G}'$.  It is an easy exercise to show that this construction has the claimed property.

For the remainder of this appendix, we assume that, first, the input graph is replaced by
its Peeters supergraph, and second, the transformation described in this subsection has been applied to the supergraph.  We therefore dispense with the notation 
$\mathcal{G}'$ for the transformed graph.

\subsection{A linear algebraic lemma}

%\begin{lem}
%Let block matrix
%$$A=\left(\begin{array}{cc} A_1 & A_2^T \\ A_2 & A_3 \end{array}\right)$$
%be positive semidefinite.   Assume $A_1\in\SS^{r}$, $A_2\in\R^{s\times r}$, $A_3\in\SS^{s}$. Then there exists %$W\in\R^{r\times s}$ such that $A_2^T=A_1W$.
%\label{lem:rangeA1}
%\end{lem}

%\begin{proof}
%Suppose for a contradiction that
%no such $W$ exists, in other words, there exists a column of $A_2^T$, say $A_2(i,:)^T$ that is not in %$\Range(A_1)$.  This implies that there exists a $\vv\in\R^r$ such that $A_1\vv=\bz$ while 
%$\gamma:=A_2(i,:)^T\vv>0$.  Consider the vector $\bar\vv\in\R^{r+s}$ whose first $r$ entries equal $\vv$ and %whose final $s$ entries are 0's except for a scalar $\alpha$ in position $i$.  Then we observe that
%\begin{align*}
%\bar\vv^TA\bar\vv&=\vv^TA_1\vv+2\vv^TA_2(i,:)^T\alpha+\alpha^2A_3(i,i) \\
%&=2\gamma\alpha + \alpha^2 A_3(i,i).
%\end{align*}
%The quantity on the preceding line can be made negative by choosing
%$\alpha<0$, $|\alpha|$ small since $\gamma>0$.  This contradicts the hypothesis that $A\succeq 0$.
%\end{proof}

\begin{lem}
Let $A\in\SS^n$ be positive semidefinite, and suppose that all of its off-diagonal entries are negative.  Then $\rank(A)\ge n-1$. 
Furthermore:
\begin{itemize}
    \item 
    If $\rank(A)=n-1$, then $\Null(A)$ is spanned by a vector all of whose entries are positive.  Furthermore, no vector all of whose entries are positive can be in $\Range(A)$.
    \item
    If $\rank(A)=n$, then all entries of $A^{-1}$ are positive.
\end{itemize}
\label{lem:negoffdiagrank}
\end{lem}

\begin{proof}
The proof is by induction on $n$.  The base case $n=1$ is obvious.  Now suppose the claim holds for $n-1$.   Let us rewrite
$$A=\left(
\begin{array}{cc}
\alpha & \wv^{\top} \\
\wv & A'
\end{array}
\right),
$$
in which $\wv<\mathbf{0}$, and $A'\in\SS^{n-1}$ is semidefinite with negative off-diagonal entries.  We also know $\alpha>0$ by Lemma~\ref{lem:rangeA1} since $\wv\ne\mathbf{0}$.  Therefore, the Schur complement $S:=A'-\wv\wv^{\top}/\alpha$ exists and is semidefinite.
 However, observe that the subtracted term $\wv\wv^{\top}/\alpha$ is positive in all entries by the sign assumption.  Since $A'$ previously had negative off-diagonal entries, it follows that $S$ must also have this property.  Therefore, by the induction hypothesis, $\rank(S)\ge n-2$, hence $\rank(A)\ge n-1$.

Next, consider two subcases based on whether $\rank(S)=n-2$ or $\rank(S)=n-1$.  
In the first subcase, by induction, the 1-dimensional null space of $S$ is spanned by some $\vv>\bz$.  
Let $\bar\vv=[-\wv^{\top}\vv/\alpha; \vv]$.  Then one checks by multiplying out that $\bar\vv$ is in the null space of $A$. Vector $\bar\vv$ has all positive entries.  Next,  for a symmetric matrix, the range space and null space are orthogonal complements.  Since the null space is spanned by an all-positive vector, it is not possible for an all-positive vector to be in the range (because the inner product of two positive vectors is positive rather than 0).

The other subcase is that $\rank(S)=n-1$.  Then one checks by multiplying out that
$$A^{-1}=\left(\begin{array}{cc}
1/\alpha + \wv^{\top}S^{-1}\wv/\alpha^2 & -\wv^{\top}S^{-1}/\alpha \\
-S^{-1}\wv/\alpha & S^{-1}
\end{array}\right)$$
By the assumptions that $\wv<\bz$ and induction hypothesis
that $S^{-1}>0$, one confirms that all four blocks of this matrix are positive.
\end{proof}

\subsection{Construction of \texorpdfstring{$B$}{B} and \texorpdfstring{$r$}{r}}

The reduction to show is $(\Pct_2)$ is NP-hard is again from graph 3-coloring and is presented in this subsection. Let $\cG$ be the graph to be tested for
3-colorability.  Based on the construction in \ref{subsec:robustcoloring},
we will assume that it is 3-colorable with robustness parameter $c$ to be
determined later.
We define
\begin{equation}
B=\left(
\begin{array}{ccc}
D & K^{\top} \\
K & A
\end{array}
\right),
\label{eq:Bdef_appdx}
\end{equation}
where the blocks $D,K,A$ will be defined in the
remainder of this subsection.  The entries of these matrices will depend on a `large' parameter $s>0$ and `small' parameter $\delta>0$.  The precise value of $\delta$ is given in \eqref{eq:deltadef} below.  The value of $s$ is not precisely specified; instead, $s$ must be chosen large enough so that several inequalities involving the other parameters hold.  All of these inequalities are lower bounds on $s$, so there is no possibility of conflicting inequalities for $s$.  \textcolor{blue}{The valid range of parameter $\epsilon$ (which appears in the statement of $(\Pct_2)$) is determined by \eqref{eq:epsbdappdx}.}

First, block $A$ is the same as the block $A$ that
was constructed before Theorem~\ref{thm_P1_tilda} used for both the $(\Pc_1)$ and $(\Pc_2)$ input constructions.  
%This matrix is $n\times n$, where we defined $n:=3|V(G)|$ in the earlier construction.
%We define $A$ as:
%$$A =\left(
%\begin{array}{ccc}
%  D & L \\
%  L^{\top} & P
%\end{array}
%\right),$$
%whose entries are as follows.
%Matrix $P$ is a modification of
%Peeters' matrix as follows.
%First, each row and column of Peeters' matrix
%is replaced with three copies of the row and column. % The diagonal
%%block is:
%$$\left(\begin{array}{ccc}
%  0 & 1 & 1 \\
%  1 & 0 & 1 \\
%  1 & 1 & 0
%\end{array}
%\right)
%$$
%Let $A^{\rm Schur}$ denote the Schur complement of %$D$, that is, the matrix
%$A$ after being updated by elimination in $D$.
%The purpose of this block is to force each diagonal %entry of $P^{\rm Schur}$
%to be equal to 1.
Thus, $A\in\R^{3n\times 3n}$, where for this appendix $n:=|\cV(\cG)|$.

As earlier, use the term ``nonedge'' to refer to a pair
$\{i,j\}\in \cK(\cV(\cG))\setminus \cE(\cG)$.  In other words, $\{i,j\}\in \cE(\bar{\cG})$, the complement of $\cG$.
Let $\bar m$ denote the number of nonedges, which is equal
to $\choose{n}{2}-|\cE(\cG)|$.

%The second modification is: $P$ will have $2$'s in all off-diagonal positions in
%which Peeters has an unspecified entry (corresponding to nonedges in the graph).
%This means that for each $i,j\in V(G)$, $P$ will have a 
%$3\times 3$ submatrix above and below
%the main diagonal blocks.  If $(i,j)$ corresponds
%to an edge, then this will be a block of nine 0's, and if $(i,j)$ is a nonedge,
%this will be a block of nine 2's.

Next, we turn to block $K$,
which has nine columns per nonedge of $\cG$, i.e., $9\bar{m}$ columns.
Focus on one of those nine
columns, say column $u$, for one particular nonedge $\{i,j\}\in \cE(\bar{\cG})$.  This row is in correspondence
with a particular `2' and its symmetric partner in $A$ (out of the nine pairs that correspond to $\{i,j\}$),
say the `2' that appears in position $(k,l)$ and $(l,k)$ of
$A$.    Then
$$
K(t,u)=\left\{
\begin{array}{ll}
s, & \mbox{if $t=k$ or $t=l$}, \\
s\delta, & \mbox{else.}
\end{array}
\right.
$$
Thus, $K$ has two kinds of entries, `small' (namely, $s\delta$) and `large' (namely $s$) with exactly two large entries per column.
In what follows,
we write $K=K^{\rm lg}+K^{\rm sm}$ with the large and small entries
respectively.

The parameter denoted ``$p(n)$'' in the statement of $(\Pct_2)$ is actually $p(9\bar{m}+3n)$ in the context of the construction of this section.  Since the argument of $p(\cdot)$ does not change in this section, we will abbreviate $p(9\bar m+3n)$ as $\hat p$.  We will assume that $\hat p>1$ throughout.

Define $D$ to be a $9\bar m\times 9\bar m$ matrix all of whose off-diagonal entries equal $-2\hat p \eps$ and whose diagonal entries are 0's.
Finally, define the rank cutoff $r:=9\bar m + 3$. 

\textcolor{blue}{The main theorem about this construction is as follows. Its proof occupies the remainder of this appendix. This theorem shows that graph 3-colorability can be reduced to $(\Pct_2)$, and thus the latter is NP-hard.}

\begin{thm}
\textcolor{blue}{
Given a graph $\cG$, form $\cG'$ as in Section~\ref{subsec:robustcoloring}, and from $\cG'$ form $B$ given by \eqref{eq:Bdef_appdx}.  If $\cG$ is 3-colorable, then there exists an $H\in\SS^{9\bar{m}+3n}$ such that $\Vert H\Vert_F\le \eps$ and a vector $\dv\in\R^{9\bar m + 3n}$ such that $\rank(B+\Diag(\dv)+H)\le 9\bar m + 3$.  On the other hand, if $\cG$ is not 3-colorable, then for any $H$ such that $\Vert H\Vert_F\le \hat p\eps$ and for any vector $\dv$, $\rank(B+\Diag(\dv)+H)\ge 9\bar m + 4$.
}
\label{thm:p2til_nphard}
\end{thm}

\subsection{Rank when \texorpdfstring{$\cG$}{G} is 3-colorable}

We argue that if $\cG$ is 3-colorable, then there exists a nonnegative $\dv\in\R^{9\bar{m}+3n}$ and perturbation
$H\in \SS^{9\bar m + 3n}$ such that $\Vert H\Vert_F\le\epsilon$ and such that
$\rank(B+H+\Diag(\dv))\le 9\bar m +3$.
Write $\dv=[\dv^\newedge,\dv^\newvertix]$, and partition
$$H=\left(\begin{array}{cc}
H_1 & H_2^{\top} \\
H_2 & H_3
\end{array}
\right)
$$
conformally with $B$.
Each nonedge of $\cG$ corresponds to nine entries of $\dv^\newedge$.  Consider diagonal entry $u\in\{1,\ldots,9\bar{m}\}$ corresponding to nonedge $\{i,j\}$.  Then we choose
$$d^\newedge(u)=
\left\{
\begin{array}{ll}
s^2/2 & \mbox{if $i$, $j$ have different colors}, \\
s^2 & \mbox{else.}
\end{array}
\right.
$$
We choose $H_1=0$, $H_2=0$, and $H_3$ will be specified below.

The goal is to show that we can choose $\dv^{\cV}$ and $H_3$ so that the Schur complement of the $(1,1)$ block of $B+H+\Diag(\dv)$, which is,
$$S:=A-K(D+\Diag(\dv^\newedge))^{-1}K^{\top}+\Diag(\dv^{\cV})+H_3,$$
is positive semidefinite and of rank 3. 
We write as $S=S'+\Diag(\dv^\cV)+H_3$ where 
$$S':=A-K(D+\Diag(\dv^\newedge))^{-1}K^{\top}.$$
 We begin by considering 
$$S'':=A-K^{\rm lg}\Diag(\dv^\newedge)^{-1}(K^{\rm lg})^{\top},$$
and below we apply Lemma~\ref{lem:KDK_lemma} to bound $\Vert S'-S''\Vert$.

Consider some entry $(k,l)$ of $A$ corresponding to nonedge $\{i,j\}$ such that $i,j$ have different colors.  The entry $(k,l)$ of $A$, which equals 2, in turn corresponds to a column $u$ of $K$.  Entry $(k,l)$ of the Schur complement update $K^{\rm lg}\Diag(\dv^\newedge)^{-1}(K^{\rm lg})^{\top}$ is determined by column $u$ of $K^{\rm lg}$ because no other column of $K^{\rm lg}$ has nonzero entries in both positions $k,l$.   
In this case, $K^{\rm lg}(k,u)d^\newedge(u)^{-1}K^{\rm lg}(l,u)=
s(s^2/2)^{-1}s=2$, hence $S''(k,l)=2-2=0$.  On the other hand, if $i,j$ have the same color, then $K^{\rm lg}(k,u)d^\newedge(u)^{-1}K^{\rm lg}(k,u)=
s(s^2)^{-1}s=1$ and hence $S''(k,l)=2-1=1$.  Therefore, all off-diagonal entries of $S''$ corresponding to nonedges in the same color class are 1's, while those corresponding to different color classes are 0's.  Entries corresponding to edges are 0's by construction and are not updated by $K^{\rm lg}\Diag(\dv^\newedge)^{-1}(K^{\rm lg})^{\top}$.  Finally, entries corresponding to the same vertex but different representatives (i.e., off-diagonal entries of the $3\times 3$ principal diagonal submatrices) are also 1's by construction and are not updated.  Thus, the off-diagonal entries of $S''$ consist of three disjoint blocks of all 1's.

Next, Lemma~\ref{lem:KDK_lemma} states that
\begin{equation}
\Vert S''-S'\Vert_F 
\le
\frac{1}{2} \Vert \Diag(\dv^\newedge)^{-1}\Vert_F\cdot \Vert K^{\rm lg}\Vert_F
\cdot \left(9 \Vert K^{\rm lg}\Vert_F \cdot \Vert \Diag(\dv^\newedge)^{-1}\Vert_F \cdot \Vert D\Vert_F
+ 5\Vert K^{\rm sm}\Vert_F\right).
\label{eq:ssprime1}
\end{equation}
Note that the two hypotheses of the lemma can be confirmed using the formula \eqref{eq:deltadef} below and the assumption that $s$ is sufficiently large.
We have the following straightforward upper bounds based on
construction of $K,\dv^\newedge,D$ and the fact that 
$\Vert A\Vert_F\le \Vert A\Vert_{\rm max}\cdot \sqrt{\mathrm{nnz}(A)}$:
\begin{align*}
    \Vert K^{\rm lg}\Vert_F &\le \sqrt{18\bar{m}}s, \\
    \Vert K^{\rm sm}\Vert_F &\le \sqrt{27\bar{m}n}s\delta, \\
    \Vert \Diag(\dv^\newedge)^{-1}\Vert_F &\le 18\bar{m}/s^2, \\
    \Vert D\Vert_F &\le 6\bar{m}\hat p \eps.
\end{align*}
Substituting these bounds in \eqref{eq:ssprime1} yields
\begin{align}
\Vert S'-S''\Vert_F 
&\le 27\sqrt{2} {\bar m}^{3/2}s^{-1}\left(15\sqrt{3\bar{m}n}s\delta+2916 \sqrt{2}\bar{m}^{5/2}\hat p s^{-1} \eps\right) \notag\\
&= 3^4 \cdot 5 \sqrt{6}{\bar m}^2n^{1/2}\delta + 2^3 \cdot 3^9 \cdot \bar{m}^{4}\hat p\eps/s^2.\label{eq:sprimesprimeprime}
\end{align}
We can ensure both terms are at most $\eps/2$ by defining
\begin{equation}
  \delta:=\frac{\epsilon}{10 \cdot 3^4 \cdot \sqrt{6} {\bar m}^2 n^{1/2}},
  \label{eq:deltadef}
\end{equation}
and choosing $s$ sufficiently large to upper bound the second term of \eqref{eq:sprimesprimeprime} also by $\eps/2$.
Since $\Vert S'-S''\Vert_F\le \eps$, we define $H_3$ to equal $S''-S'$ in off-diagonal entries, and equal to 0 on the diagonal.   In this way $S'+H_3$ is a matrix whose off-diagonal entries are 0's and 1's, and the 1's are arranged in three disjoint square blocks.  Observe that the diagonal entries of $S'+H_3$ are all negative since the diagonal entries of $A$ and $H_3$ are 0's, while the diagonal entries of $-K(D+\Diag(\dv^\newedge))^{-1}K^{\top}$ are negative.  Therefore, there is some positive vector $\dv^\newvertix$ such that the diagonal entries of $S'+H_3+\Diag(\dv^\newvertix)$ are all 1's.  This matrix is positive semidefinite, its rank is 3, and it is the Schur complement after eliminating the $(1,1)$ block of $B+\Diag(\dv)+H$.

\subsection{Rank when \texorpdfstring{$\cG$}{G} is not 3-colorable}

The hypothesis of this
subsection is that $\cG$ is not 3-colorable, and therefore, by robustness described in \ref{subsec:robustcoloring},
neither is any induced subgraph with at least $|\cV(\cG)|-c$ nodes, where $c$
is to be determined.  We 
prove that $\rank(B)\ge 9\bar m+4$ for any perturbation of size $\hat{p}\cdot \eps$
plus any diagonal
matrix that yields a positive semidefinite matrix.  For this section, let $H$ be a perturbation to $B$, and let $\dv$ be a vector, and suppose that $\Vert H\Vert_F\le\hat{p}\eps$ and $B+H+\Diag(\dv)\succeq 0$.  We will argue that $\rank(B+H+\Diag(\dv))\ge 9\bar m + 4$.

As in the previous section, partition $H$ and $\dv$ conformally with $B$.  The first aim of this subsection is to show that the $(1,1)$ block of $B+H+\Diag(\dv)$, namely, $D+H_1+\Diag(\dv^\newedge)$ is invertible.  Once this is proved, the main task of this subsection is to show that rank of the Schur complement, that is, the rank of 
\begin{equation}
S:=A+H_3-(K+H_2)(D+\Diag(\dv^\newedge)+H_1)^{-1}(K+H_2)^{\top}+\Diag(\dv^V)
\label{eq:Schur1}
\end{equation}
is at least 4.  

We first observe that since $\Vert H_1\Vert_F\le \hat p\eps$, all off-diagonal entries of $D+\Diag(\dv^\newedge)+H_1$ are negative because $H_1$ is not large enough to cancel the negative entries of $D$.   This means that $\rank(D+\Diag(\dv^\newedge)+H_1)\ge 9\bar{m}-1$ by Lemma~\ref{lem:negoffdiagrank}.  In fact, we claim more strongly that $\rank(D+\Diag(\dv^\newedge)+H_1)= 9\bar{m}$.  Consider summing all columns of $K^{\top}$ to yield a $9\bar{m}$-length vector $K^{\top}\mathbf{e}$.  Since each row of $K^{\top}$ has exactly two entries equal to $s$ and $3n-2$ entries equal to $s\delta$, each entry of $K^{\top}\mathbf{e}$ is exactly $(2+(3n-2)\delta)s$. 
%Note that $\delta$ is sufficiently small by \eqref{eq:deltadef} to assure that 
Thus, each entry of $K^{\top}\mathbf{e}$ is at least $2s$.  Then each entry of $(K+H_2)^{\top}\mathbf{e}$ is at least $1.9s$ since $\Vert H_2\Vert_F\le \eps$, and $\eps\le 1$ is much smaller than $s$ (i.e., choose $s$ large enough to ensure this).  Since $B+H+\Diag(\dv)$ is positive semidefinite, it follows from Lemma~\ref{lem:rangeA1} that $(K+H_2)^{\top}\mathbf{e}\in \Range(D+\Diag(\dv^\newedge)+H_1)$.  Since this vector $(K+H_2)^{\top}\mathbf{e}$ has all positive entries, it then follows from Lemma~\ref{lem:negoffdiagrank} that $\rank(D+\Diag(\dv^\newedge)+H_1)=9\bar{m}$.

Thus, for the rest of this analysis, $S$ given by \eqref{eq:Schur1} is well-defined, and we define $\Pi:=(D+\Diag(\dv^\newedge)+H_1)^{-1}$, which is a factor appearing in the formula for $S$.  By Lemma~\ref{lem:negoffdiagrank}, $\Pi>0$.  As noted in Section~\ref{sec:np-hard}, the hurdle in the analysis of $B+H+\Diag(\dv)$ is that there is no prior lower bound on the entries of $\dv^\newedge$, in other words, no prior upper bound on $\Vert \Pi\Vert_F$.  Therefore, we take cases depending on the sizes of diagonal entries of $\Pi$.  (Recall that the largest entry of a positive semidefinite matrix always appears on the diagonal.)

\underline{Case 1, For at least one $u^*\in\{1,\ldots,9\bar{m}\}$, 
$\Pi(u^*,u^*)\ge  \pi_1$}. Here,
\begin{equation}
    \pi_1=
\frac{2\hat p\eps}{s^2\delta^2}.
\label{eq:pi1_def}
\end{equation}

Let us denote the entry of $A$ corresponding to $u^*$ as $A(k,l)$.  Therefore, $A(k,l)=2$, and $(k,l)$ corresponds to a nonedge $\{i,j\}$ of $\cG$.  We claim entries in columns $k$ and $l$ of $A-(K+H_2)\Pi(K+H_2)^{\top}$ are negative.  The update $(K+H_2)\Pi(K+H_2)^{\top}$ may be written as a sum of rank-one matrices:
$$(K+H_2)\Pi(K+H_2)^{\top}=\sum_{u=1}^{9\bar m}\sum_{u'=1}^{9\bar m}(K+H_2)(:,u)\Pi(u,u')(K+H_2)(:,u')^{\top}.$$
Consider one term, namely,
$(K+H_2)(:,u^*)\Pi(u^*,u^*)(K+H_2)(:,u^*)^{\top}$.  Note that $(K+H_2)({:,u^*})$ has entries at least $0.9s$ in positions $k,l$ and entries at least $0.9s\delta$ in the remaining positions because $K$ has entries $s$ and $s\delta$ in these positions, and $H_2$ is much smaller assuming $s$ is chosen sufficiently large.  Therefore, every entry in row $k$ of 
$(K+H_2)({:,u^*})\Pi(u^*,u^*)(K+H_2)({:,u^*})^{\top}$ is at least $0.9s\cdot\pi_1\cdot 0.9s\delta=1.62s^2\delta(\hat p\eps)/(s^2\delta^2)=1.62\hat p\eps/\delta>3$, by assumption that $\hat p\ge 1$ and $\eps/\delta>10$ by \eqref{eq:deltadef}. This is just one term in the update, but since all $(9\bar m)^2$ terms are positive matrices, the remaining terms can only further increase the overall product.  Thus, since $A(k,l')\le 2$ for all $l'=1,\ldots,3n$, and the $(k,l')$ entry of the update is at least 3, we conclude that the $(k,l')$ entry of $A-(K+H_2)\Pi(K+H_2)^{\top}$ is bounded above by $-1$.  This same analysis applies to the $(l,l')$ entry for all $l'=1,\ldots,3n$.

Next, consider an entry $A(k',l')$ corresponding to an edge $\{i',j'\}$ in $\cG$.  By our construction, $A(k',l')=0$.  Observe that entry $(k',l')$ of $(K+H_2)\Pi(K+H_2)^{\top}$ is at least $(0.9s\delta)^2\pi_1=1.62s^2\delta^2(\hat p\eps)/(s^2\delta^2)=1.62\hat p\eps$.  This is seen by considering the same term used in the previous paragraph, namely, entry $(k',l')$ of $(K+H_2)(:,u^*)\Pi(u^*,u^*)(K+H_2)(:,u^*)^{\top}$, noting that the other terms can only increase the update.   Thus, the $(k',l')$ entry of $A -(K+H_2)\Pi(K+H_2)^{\top}$ is bounded above by $-1.62\hat p\eps$. 

Since the off-diagonal entries of $S$ (given by \eqref{eq:Schur1}) equal $A-(K+H_2)\Pi(K+H_2)^{\top}+H_3$ and $\Vert H_3\Vert_F\le \hat p\eps$, we conclude that all entries of $S$ analyzed in the previous few paragraphs are negative.  In other words, all entries of $S$ in columns $k$ and $l$ and all off-diagonal entries corresponding to edges of $\cG$ are negative.  

Find a 3-cycle in $\cG$ that does not include either $i$ or $j$, say
$i_1,i_2,i_3\in \cV(\cG)$.  Note that $\cG$ has many 3-cycles if we first
apply the robustifying transformation described in \ref{subsec:robustcoloring}, so certainly
such a cycle can be found. Each of these three nodes corresponds in turn
to three
rows of $A$; let $l_1,l_2,l_3$ be three such
rows of $A$ (choose among the three representatives
of each of $i_1,i_2,i_3$ arbitrarily).  Now consider the $5\times 5$
principal submatrix
of $S$ indexed by $k,l,l_1,l_2,l_3$.  All the off-diagonal entries in
this matrix are negative as argued in the previous
paragraph.  Therefore, the rank of this submatrix is at least 4 by Lemma~\ref{lem:negoffdiagrank}.
This concludes the analysis of Case 1.

\underline{Case 2, $\Pi(u,u')<\pi_1$, for all $u,u'=1,\ldots,9\bar{m}$.}  As in Case 1, $\pi_1$ is given by \eqref{eq:pi1_def}. 

The assumption of this case allows us to derive a stronger upper bound than $\pi_1$ on the off-diagonal entries of $\Pi$.  We already know that $\Pi$ is positive definite and positive.  Recall from the definition of $\Pi$ that $\Pi(D+\Diag(\dv^\newedge)+H_1)=I$.  Focusing on a diagonal entry of this product, we know that for each $u=1,\ldots,9\bar m$, 
$$\sum_{u'=1}^{9\bar m} \Pi(u,u')(D(u,u')+H_1(u,u')) + \Pi(u,u)d(u) = 1.$$
Since $\Pi(u,u')\le\max(\Pi(u,u),\Pi(u',u'))<\pi_1=2\hat p\eps/(s^2\delta^2)$, $|D(u,u')|\le 2\hat p\eps$, and $|H_1(u,u')|\le \hat p\eps$, we conclude that each term in the above summation is at most $6\hat p^2\eps^2/(s^2\delta^2)$ in magnitude.
By choosing $s$ sufficiently large,
we are assured that every term in the summation is at most $1/(90\bar m)$ in magnitude.  Then,
we conclude that the first summation is at most $1/10$ in magnitude, and therefore $\Pi(u,u)d(u)\ge 0.9$ so $d(u)\ge 0.9/\pi_1$ for each $u=1,\ldots,9\bar m$.

Now consider an off-diagonal entry of the product
$\Pi(D+\Diag(\dv^\newedge)+H_1)$, say entry $(u,u')$ to obtain
$$\sum_{u''=1}^{9\bar m} \Pi(u,u'')(D(u,u'')+H_1(u,u''))+\Pi(u,u')d(u')=0.$$
The magnitude of each term in the summation is at most $6\hat p^2\eps^2/(s^2\delta^2)$ as in the last paragraph.  By choosing $s$ sufficiently large, we can ensure that this is at most $\delta^2/(90\bar m)$,
so the summation has magnitude at most $\delta^2/10$.   Also, $d(u')\ge 0.9/\pi_1$ as shown in the last paragraph.  Therefore,
$$\Pi(u,u') \le \frac{\delta^2}{10\cdot 0.9/\pi_1},$$
for each $u,u'=1,\ldots,9\bar m$ with $u\ne u'$, which
implies that  $\Pi(u,u')\le \delta^2\pi_1/9$,
which is stronger than the upper bound of $\pi_1$ that holds by the assumption of this case.

Next, let us define two ranges of $\Pi(u,u)$ for $u=1,\ldots,9\bar m$.
\begin{itemize}
    \item
    {\em Small} entries satisfy
    $\Pi(u,u)\in(0,\pi_0]$.
    \item 
    {\em Large} entries satisfy
    $\Pi(u,u)\in (\pi_0,\pi_1)$.
\end{itemize}
Here, 
\begin{equation}
\pi_0=\frac{\eps_0}{27 \bar m n s^2\delta}.
\label{eq:pi0def}
\end{equation}
In this formula, $\eps_0$ is the universal constant introduced in Theorem~\ref{thm:tilde_p3_is_nphard}.
Note that the case that $\Pi(u,u)>\pi_1$ (which could be called ``very large'') was already handled
by Case 1.

We now form a graph $\cG_{\rm lg}$ defined as follows.  The nodes $\newvertix(\cG_{\rm lg})$ of the graph are integers $1,\ldots,3n$ in correspondence with the rows/columns of $A$.  Given a node in $\cG$,
we will say that it ``owns'' the three nodes of $\cG_{\rm lg}$ associated
with that original node.
Let $k,l$ be a pair of nodes of $\cG_{\rm lg}$, $k\ne l$, that correspond to a nonedge of $\cG$.  In this case, there is an index $u\in\{1,\ldots,9m\}$ that corresponds to the pair $(k,l)$.  Include the edge $\{k,l\}$ in $\cE(
\cG_{\rm lg})$ whenever $\Pi(u,u)$ is large according to the above dichotomy.

We now take two subcases of Case 2.

\underline{Case 2a, the edges of $\cG_{\rm lg}$ can
  be covered by $c$ nodes of $\cG_{\rm lg}$.}
  
Here, $c$ is the constant that makes non-colorability of $\cG$ robust; the precise value of $c$ will be selected in Case 2b below.  Note that this case includes the subcase that $\cG_{\rm lg}$ has no edges, i.e., there are no large diagonal entries of $\Pi$.

Let this vertex cover be denoted $\Sigma_c\subseteq \newvertix(\cG_{\rm lg})$.
Let $\Sigma\subseteq \cE(\newvertix(\cG))$ be the at most $c$
nodes of $\cG$ that collectively own $\Sigma_c$,
and let $\Sigma_{\rm lg}$ be the at most $3c$ nodes of $\cG_{\rm lg}$ owned by $\Sigma$.
By construction $\Sigma_c\subseteq \Sigma_{\rm lg}$.  

Let $\bar \Sigma=\newvertix(\cG)\setminus \Sigma$ and let $\bar \Sigma_{\rm lg}$ be the $3|\bar \Sigma|$ nodes
of $\cG_{\rm lg}$ (equivalently, rows of $A$) owned by $\bar \Sigma$.
Since $\cG$ is robustly not 3-colorable with parameter $c$,
the induced subgraph $\cG[\bar \Sigma]$ is not 3-colorable.  
Recall that $S$ stands for the Schur complement of the $(1,1)$ block of $B+H+\Diag(\dv)$.  We will argue that $S(\bar \Sigma_{\rm lg},\bar\Sigma_{\rm lg})$ may be written as $A+L+H'$, where $A,L,H'$ satisfy conditions 2(a)--(d) of Theorem~\ref{thm:tilde_p3_is_nphard}.  Since $G[\bar \Sigma]$ is not 3-colorable, it follows from the theorem that $\rank(S(\bar \Sigma_{\rm lg}, \bar\Sigma_{\rm lg}))\ge 4$, which in turn implies $\rank(S)\ge 4$, thus concluding Case 2a.   The matrix $A$ in this formula for $S(\bar \Sigma_{\rm lg},\bar\Sigma_{\rm lg})$ is the same as the matrix $A$ defined earlier in the appendix, all of whose entries are 0's, 1's, or 2's.

Recall from the notation of the theorem that $L$ is may have arbitrary entries in positions corresponding to nonedges of $\cG[\bar\Sigma]$ and in diagonal positions.  Therefore, these entries of $S(\bar \Sigma_{\rm lg},\bar \Sigma_{\rm lg})$ do not need any further analysis in
our decomposition of $S(\bar \Sigma_{\rm lg},\bar \Sigma_{\rm lg})$ as the sum $A+L+H'$. The entries of $S$ that need consideration are therefore the entries corresponding to edges (which are 0's in $A$) and the six off-diagonal entries of the $3\times 3$ diagonal blocks (which are 1's in $A$).  

Therefore, we now proceed to bound the entries of the update to $A$ 
$$H':=H_3 - (K+H_2)\Pi(K+H_2)^{\top}$$
that correspond to edges of $\cG[\bar \Sigma]$ and to off-diagonal entries of diagonal blocks of $S(\bar \Sigma_{\rm lg},\bar\Sigma_{\rm lg})$.   Let us call these the ``distinguished" entries of $H'$.
As above, let us write the second term of $H'$ as a sum of $(9\bar m)^2$ rank-one matrices:
$$(K+H_2)\Pi(K+H_2)^{\top}=\sum_{u=1}^{9\bar m}\sum_{u'=1}^{9\bar m} (K+H_2)(:,u)\Pi(u,u')(K+H_2)(:,u')^{\top}$$
and classify the terms into three categories: (a) terms such that $u=u'$ and $\Pi(u,u)$ is large, (b) terms such that $u=u'$ and $\Pi(u,u')$ is small, and (c) terms such that $u\ne u'$.

For $u$ in category (a), a large diagonal entry of $\Pi$, we know $\Pi(u,u)\le \pi_1$ by the hypothesis of Case 2.  Furthermore, we know that entries of $K(:,u)K(:,u)^{\top}$ in distinguished positions are
$(s\delta)^2$ (that is, small entries of $K$ squared)  because the two large entries of column $K(:,u)$, say in positions $k,l$, are such that both $k,l$ are excluded from $\bar\Sigma_{\rm lg}$ by construction of $\bar\Sigma_{\rm lg}$.    The perturbation $H_2$ raises this to at most $2.2(s\delta)^2$ (since $\Vert H_2\Vert_F\le \hat p\eps$, and we can make $s$ sufficiently large). 
Therefore, if $(k,l)$ is a distinguished entry and $u$ is in category (a), we have shown that
\begin{align*}
((K+H_2)(:,u)\Pi(u,u)(K+H_2)(:,u)^{\top}))(k,l) & \le
2.2(s\delta)^2\pi_1 \\
&=2.2\hat p\eps.
\end{align*}

Next, if $u$ is in category (b), a small diagonal entry of $\Pi$, we know that $\Pi(u,u)\le \pi_0$. Furthermore, we know that entries of $K(:,u)K(:,u)^{\top}$ in distinguished positions are either $(s\delta)^2$ or $s^2\delta$, i.e., a product of two small entries of $K$ or a small and large entry.  This is because the products of two large entries appear only in positions $(k,l)$ and $(l,k)$, where $k$ and $l$ are the indices of the row associated with a nonedge by construction of $K$.  However, entries corresponding to nonedges are not distinguished in this analysis.  Therefore, an upper bound on this entry is $s^2\delta$, and therefore an upper bound on the corresponding entry of $(K+H_2)(:,u)(K+H_2)(:,u)^{\top}$ is $1.1s^2\delta$.  Therefore, if $(k,l)$ is a distinguished entry and $u$ is in category (b), we have show that
\begin{align*}
((K+H_2)(:,u)\Pi(u,u)(K+H_2)(:,u)^{\top})(k,l) &\le 1.1s^2\delta\pi_0 \\
&=\frac{0.1\eps_0}{27\bar m n}.
\end{align*}

Finally, if $(u,u')$ is in category (c), i.e., an off-diagonal entry of $\Pi$, then entries of $K({:,u})K({:,u'})^{\top}$ can be as large as $s^2$, and $\Pi(u,u')\le \pi_1\delta^2/9$ by the analysis at the beginning of Case 2, so for any entry $(k,l)$ of the product (distinguished or not),
\begin{align*}
((K+H_2)(:,u)\Pi(u,u')(K+H_2)(:,u')^{\top})(k,l) &\le 1.1s^2\cdot \pi_1\delta^2/9 \\
&\le 0.3\hat p \epsilon.
\end{align*}

In all three cases we have an upper bound on the contribution of one rank-one term to a distinguished entry of $H'$.  The contributions from category (b) add up to at most $0.1\eps_0/(3n)$ since there are at most $9\bar{m}$ category (b) terms.  The sum of terms from categories (a) and (c)  is at most $2.2\cdot(9\bar m)^2\hat p\eps$ since there are at most $9\bar m$ category-(a) terms and at most $(9\bar m)^2-9\bar m$ category-(c) terms.  We can place an upper bound of $0.9\eps_0/(3n)$ on the sum of terms of categories (a) and (c) if we impose the assumption:
\begin{equation}
\eps \le \frac{\eps_0}{600\bar m^2 n \hat p}.
\label{eq:epsbdappdx}
\end{equation}
With this assumption in place, we can now claim that the sum of 
all contributions to a distinguished entry of $H'$ is at most $\eps_0/(3n)$ in magnitude.  The non-distinguished entries of $H'$ may be set to 0 since these entries correspond to edges (which are covered by $L$ or $\Diag(\dv^{\cV})$).   Since the size of $H'$ is $3n\times 3n$, it follows that $\Vert H'\Vert_F\le \eps_0$, and thus all hypotheses 2(a)--(d) of Theorem~\ref{thm:tilde_p3_is_nphard} are satisfied except the rank assumption.  Since the graph is not 3-colorable, this implies that 
$\rank(S(\bar\Sigma_{\rm lg},\bar\Sigma_{\rm lg}))\ge 4$.
This concludes the analysis of Case 2a.

\underline{Case 2b, more than $c$ nodes  of $\cG_{\rm lg}$ are required to cover the edges of $\cG_{\rm lg}$.}

We will now specify $c=5$.

Let $u$ be a large diagonal entry of $\Pi$ corresponding to entry $(k,l)$ of $A$.  This entry also corresponds to an edge of $\cG_{\rm lg}$ and to a nonedge $\{i,j\}$ of the original graph $\cG$.  We have the following lower bound:
\begin{align*}((K+H_2)\Pi(K+H_2)^{\top})(k,l)&\ge 0.9s^2\pi_0\\
&= \frac{\eps_0}{30\bar m n \delta}\\
&\ge  \frac{\eps_0 \cdot 27 \sqrt{6} \bar m} { n^{1/2} \epsilon} \\
&\ge 16200 \sqrt{6}\bar m^3 n^{1/2}\hat p\\
& > 3.96\cdot 10^4 \cdot \bar m^3 n^{1/2}\hat p.
\end{align*}
The first line is obtained because $K(k,u)=K(l,u)=s^2$ (i.e., large entries of $K$), $H_2$ does not perturb this much assuming $s$ is chosen sufficiently large, and $\Pi(u,u)\ge \pi_0$ by assumption that $u$ is a large diagonal entry.  Contributions from the other entries of $\Pi$ can only increase this since $K+H_2$ and $\Pi$ are both positive matrices.
The second line follows from \eqref{eq:pi0def}, the third from
\eqref{eq:deltadef}, and the fourth from 
\eqref{eq:epsbdappdx}.  Thus, 
\begin{align*}
    S(k,l) &= A(k,l)+H_3(k,l)-((K+H_2)\Pi(K+H_2))(k,l) \\
    & < 2 + \eps - 3.96\cdot 10^4\\
    & < -3.95 \cdot 10^4.
\end{align*}

Let us rename $S$ as $S^{(0)}$ and consider performing the following operations for $\mu=0,1,2,\ldots$
\begin{enumerate}
    \item 
    Let $\tilde S^{(\mu)}\in\SS^{3n-\mu}$ denote $PS^{(\mu)}P^{\top}$, where $P$ is a permutation matrix chosen so that the largest diagonal entry of $\tilde S^{(\mu)}$ is in the $(1,1)$ position.  
    \item
    Terminate if $\tilde S^{(\mu)}(1,1)=0$.
    \item
    Else let $S^{(\mu+1)}\in \SS^{3n-\mu-1}$ be the Schur complement of the $(1,1)$ entry of $\tilde S^{(\mu)}$.
\end{enumerate}
Clearly the sequence of matrices produced by this iteration are all symmetric and positive semidefinite by the Schur complement lemmas.
Recall the following fact: the entry in a positive semidefinite matrix with the largest magnitude must occur on the diagonal.  Therefore, $\tilde S^{(\mu)}(1,1)$ is the entry with largest magnitude of $\tilde S^{(\mu)}$.  This in turn means that step 2 will not terminate provided that $S^{(\mu)}$ has any nonzero entry.

We claim that we can perform these operations at least 4 times before
termination in step 2.  Assuming we prove this claim, this implies that $\rank(S)\ge 4$ by the Schur complement lemmas, and this would therefore conclude the analysis of Case 2b.

To prove that we can perform the above iteration at least 4 times, we establish the following claims by induction.

{\bf Claims:}
\begin{enumerate}
    \item 
    The largest positive off-diagonal entry in $S^{(\mu)}$ is at most $(2+\eps)^{\mu+1}$.
    \item
    Let $r_\mu:=(2+\eps)^{\mu+1} - 10^4$.   Call the entries of $S^{(\mu)}$ whose value is less than $r_\mu$ the ``big negative entries.''  The big negative entries cannot be covered by fewer than $c+1-\mu$ rows/columns of $S^{(\mu)}$.
\end{enumerate}
Once we prove these claims by induction, then the main result follows because $r_\mu<0$ for $\mu=0,1,2,3,4$ and $c+1-\mu>0$ for $\mu=0,1,2,3,4$, so the second claim implies that there is at least one nonzero entry in $S^{(\mu)}$, and hence step 2 of the above algorithm will not terminate.

The base of the induction is as follows.  Recall that
$$S^{(0)}=S=A+H_3-(K+H_2)\Pi(K+H)^{\top}+\Diag(\dv^\newvertix).$$
The maximum entry of $A$ is 2, the maximum entry of $H_3$ is $\eps$, the third term is negative, and the fourth term does not affect off-diagonal entries.  Thus, the first induction claim holds for $\mu=0$.  We argued earlier that $S^{(0)}$ contains an entry smaller than $-10^4$ in position $(k,l)$ for every $(k,l)$ corresponding to a large diagonal entry of $\Pi$, i.e., edges of $\cG_{\rm lg}$.  The hypothesis for Case 2b is that these entries cannot be covered by fewer than $c+1$ rows/columns of $S^{(0)}$, i.e., nodes of $\cG_{\rm lg}$.  Thus, the second induction hypothesis also holds for $\mu=0$.

Now assume the hypothesis holds for $\mu$; we show that it holds for $\mu+1$.  Observe that $S^{(\mu+1)}=\tilde S^{(\mu)}(2:3n-\mu,2:3n-\mu)-\wv\wv^{\top}/\tilde S^{(\mu)}(1,1)$, where $\wv=\tilde S^{(\mu)}(2:3n-\mu,1)$.  Consider an off-diagonal entry of $S^{(\mu+1)}$, say entry $(i,j)$ with $i\ne j$.
We have: $S^{(\mu+1)}(i,j)=\tilde S^{(\mu)}(i+1,j+1)-w_iw_j/S^{(\mu)}(1,1)$.
There are two cases: either $w_i w_j\ge 0$ or $w_iw_j<0$.  In the first case, $S^{(\mu+1)}(i,j)\le S^{(\mu)}(i+1,j+1)$.  Since the bounds appearing in the induction claims are both upper bounds on the entries of $S^{(\mu)}$ and these upper bounds increase with $\mu$, then a decrease in an entry can only further sharpen the induction claim.

On the other hand, if $w_iw_j<0$, say, without loss of generality, that $w_i<0$ and $w_j>0$, then $-\tilde S^{(\mu)}(1,1)\le w_i<0$ since, as noted earlier, the largest magnitude entry of $\tilde S^{(\mu)}$ appears in the $(1,1)$ position.  Therefore, $|w_iw_j/\tilde S^{(\mu)}(1,1)|\le |w_j|\le (2+\eps)^\mu$, the second inequality arising from induction claim 1 since $w_j>0$.  Thus, $S^{(\mu+1)}(i,j)\le \tilde S^{(\mu)}(i+1,j+1)+(2+\eps)^\mu$.  Combined with the induction hypothesis, this means that the largest positive value in $S^{(\mu+1)}$ in an off-diagonal position is at most $(2+\eps)^{\mu+1}$, thus establishing the first induction claim.  Similarly, since $r_{\mu+1}=r_\mu+(2+\eps)^\mu$, the big negative entries of $\tilde S^{(\mu)}$ remain big and negative in $S^{(\mu+1)}$.

To finish the induction, we also need to show that the covering number of the big negative entries of $S^{(\mu+1)}$
is at most $c+1-(\mu+1)$.  This follows because the preceding paragraphs show that all big negative entries of $\tilde S^{(\mu)}$ persist in $S^{(\mu+1)}$ except for those lost in the discarded first row and column of $\tilde S^{(\mu)}$.  Therefore, since the big negative entries of $\tilde S^{(\mu)}$ needed at least $c+1-\mu$ rows/columns to be covered, then those of $S^{(\mu+1)}$ need at least $c+1-\mu-1$.  This concludes the induction and hence the analysis of Case 2b.